\documentclass{amsart}
\usepackage{cite,geometry,amssymb,amstext,amsmath,amscd,amsthm,amsfonts,enumerate,graphicx,latexsym,stmaryrd,multicol,bm,xcolor}
\usepackage[all]{xy}
\usepackage{comment}
\newtheorem{thm}{Theorem}[section]
\newtheorem{lem}[thm]{Lemma}
\newtheorem{prop}[thm]{Proposition}
\newtheorem{cor}[thm]{Corollary}
\theoremstyle{definition}
\newtheorem{dfn}[thm]{Definition}
\newtheorem{ques}[thm]{Question}

\newtheorem{eg}[thm]{Example}

\theoremstyle{remark}
\newtheorem{rem}[thm]{Remark}

\newtheorem*{ac}{Acknowledgments}
\newtheorem*{conv}{Convention}
\numberwithin{equation}{thm}
\def\add{\operatorname{add}}
\def\ann{\operatorname{ann}}

\def\C{\mathcal{C}}
\def\c{\mathsf{C}}
\def\cm{\operatorname{CM}}
\def\codim{\operatorname{codim}}

\def\cx{\operatorname{cx}}
\def\Db{\operatorname{\mathsf{D^b}}}
\def\depth{\operatorname{depth}}
\def\Dsg{\operatorname{\mathsf{D^{sg}}}}
\def\E{\mathcal{E}}
\def\EE{\operatorname{E}}
\def\edim{\operatorname{edim}}

\def\Ext{\operatorname{Ext}}
\def\fd{\operatorname{fd}}
\def\ge{\geqslant}
\def\gor{\operatorname{Gor}}
\def\Gproj{\operatorname{\mathsf{Gproj}}}
\def\GProj{\operatorname{\mathsf{GProj}}}
\def\grade{\operatorname{grade}}

\def\height{\operatorname{ht}}
\def\Hom{\operatorname{Hom}}

\def\id{\operatorname{id}}
\def\Im{\operatorname{Im}}
\def\inj{\operatorname{inj}}
\def\ipd{\operatorname{IPD}}
\def\K{\mathrm{K}}
\def\Ktac{\mathrm{K}_{\operatorname{tac}}}
\def\L{\mathbf{L}}
\def\le{\leqslant}

\def\M{\mathcal{M}}
\def\m{\mathfrak{m}}
\def\Max{\operatorname{Max}}
\def\Mod{\operatorname{\mathsf{Mod}}}
\def\mod{\operatorname{\mathsf{mod}}}
\def\n{\mathfrak{n}}
\def\nf{\operatorname{NF}}
\def\ng{\operatorname{Ngor}}
\def\nsgr{\operatorname{Nsgr}}
\def\pd{\operatorname{pd}}
\def\Proj{\operatorname{\mathsf{Proj}}}
\def\proj{\operatorname{\mathsf{proj}}}
\def\p{\mathfrak{p}}

\def\q{\mathfrak{q}}
\def\RHom{\operatorname{\mathbf{R}Hom}}
\def\reg{\operatorname{Reg}}
\def\sing{\operatorname{Sing}}
\def\spec{\operatorname{Spec}}

\def\syz{\Omega}
\def\T{\mathcal{T}}
\def\thick{\operatorname{thick}}
\def\Tor{\operatorname{Tor}}
\def\TR{(\mathbf{tr})}
\def\U{\mathbf{U}}
\def\V{\mathrm{V}}
\def\VV{\mathbf{V}}
\def\X{\mathcal{X}}

\def\Y{\mathcal{Y}}
\def\Z{\mathcal{Z}}
\def\ZZ{\mathbb{Z}}

\begin{document}
\allowdisplaybreaks
\title{On strongly G-regular rings}
\author{Kaito Kimura}
\address{Department of Mathematics, Purdue University, 150 N. University Street, West Lafayette, IN 47907, USA}
\email{m21018b@gmail.com}
\author{Yuki Mifune}
\address{Graduate School of Mathematics, Nagoya University, Furocho, Chikusaku, Nagoya 464-8602, Japan}
\email{yuki.mifune.c9@math.nagoya-u.ac.jp}
\author{Yuya Otake}
\address{Graduate School of Mathematical Sciences, The University of Tokyo, 3-8-1 Komaba, Meguro-ku, Tokyo 153-8914, Japan}
\email{yuyaotake31@g.ecc.u-tokyo.ac.jp}
\author{Ryo Takahashi}
\address{Graduate School of Mathematics, Nagoya University, Furocho, Chikusaku, Nagoya 464-8602, Japan}
\email{takahashi@math.nagoya-u.ac.jp}
\urladdr{https://www.math.nagoya-u.ac.jp/~takahashi/}
\thanks{2020 {\em Mathematics Subject Classification.} 13C60, 13D09, 13H10, 16E65}
\thanks{{\em Key words and phrases.} Gorenstein projective module, G-regular ring, strongly G-regular ring, quasi-dominant ring, CM-finite algebra, CM-free algebra, weakly Gorenstein algebra, virtually Gorenstein algebra, derived category}
\begin{abstract}
A noetherian ring is called G-regular when all finitely generated Gorenstein projective modules are projective. 
In this paper, we study rings satisfying the stronger condition that all Gorenstein projective modules are projective, which we call strongly G-regular. 
We show that the notion of strongly G-regular rings is closely related to that of quasi-dominant rings introduced by Takahashi and to the covariant/contravariant finiteness of a certain thick subcategory. 
We also answer a series of questions due to Chen in the negative, showing that the Gorenstein projective analogue of the Auslander--Ringel--Tachikawa theorem fails even for commutative local artin algebras which are weakly Gorenstein in the sense of Ringel and Zhang.
\end{abstract}
\maketitle
%%%%%%%%%%%%%%%%%%%%%%%%%%%%%%%%%%%%%%%%%%%%%%%%%%%%%%%%%%%%
\section{Introduction}

Let $R$ be a commutative noetherian local ring with residue field $k$ admitting a dualizing complex $D$.
Denote by $\Db(R)$ the bounded derived category of finitely generated $R$-modules.
We say that $R$ is {\em quasi-dominant} if $k$ belongs to the smallest thick subcategory of $\Db(R)$ containing $R,D$ and any non-perfect $R$-complex.
Takahashi \cite{T23} obtains a classification theorem of thick subcategories over a Cohen--Macaulay ring which is locally quasi-dominant.
We extend this theorem to the case where the ring is not Cohen--Macaulay:

\begin{thm}[Corollary \ref{lqdom}]\label{main2}
Let $R$ be a commutative noetherian ring with a dualizing complex $D$.
Suppose that $R$ is locally quasi-dominant.
Then the assignment of each subcategory $\X$ of $\Db(R)$ to the set of prime ideals $\p$ such that $X_\p$ has finite projective dimension for all $X\in\X$ induces a one-to-one correspondence
$$
{\left\{
\begin{matrix}
\text{Thick subcategories of $\Db(R)$}\\
\text{containing $R$ and $D$}
\end{matrix}
\right\}}
\cong
{\left\{
\begin{matrix}
\text{Generalization-closed subsets of $\spec R$}\\
\text{between $\reg R$ and $\gor R$}
\end{matrix}
\right\}.}
$$
\end{thm}

\noindent
Here, $\reg R$ and $\gor R$ stand for the regular and Gorenstein loci of $R$, respectively.

Takahashi \cite{T08} calls a commutative noetherian ring $R$ {\em G-regular} if any finitely generated Gorenstein projective $R$-module is projective.
Following this, we call $R$ {\em strongly G-regular} if any (possibly infinitely generated) Gorenstein projective $R$-module is projective.
Evidently, a strongly G-regular ring is G-regular.
A theorem of Iyengar and Krause \cite{IK06} implies that $R$ is strongly G-regular if and only if the smallest thick subcategory of $\Db(R)$ that contains $R,D$ coincides with $\Db(R)$.
%This enables us to reduce each problem on arbitrary Gorenstein projective modules to a problem in $\Db(R)$, where one works in a finitely generated setting.
%Another advantage is that thick subcategories and Verdier quotients behave well under localization at a prime ideal.
Using this, we can show that, under mild assumptions, strong G-regularity is a local property. 
Since G-regularity need not be preserved under localization, these two notions clearly differ in positive Krull dimensions. 
Strong G-regularity is closely related to quasi-dominance:

\begin{thm}[Corollary \ref{sgrprprng}]\label{main3}
Let $R$ be an excellent ring with a dualizing complex.
The following hold.
\begin{enumerate}[\rm(1)]
\item
The ring $R$ is strongly G-regular if and only if $R$ is locally quasi-dominant with $\reg R=\gor R$.
\item
If $R$ is a non-Gorenstein quasi-dominant local ring with an isolated singularity, $R$ is strongly G-regular.
\end{enumerate}
\end{thm}

G-regularity connects to approximation theory for finitely generated Gorenstein projective modules and to the homological finiteness of the (sub)category they form, namely, its {\em covariant/contravariant finiteness}. 
By the maximal Cohen--Macaulay approximation theorem due to Auslander and Buchweitz \cite{AB}, the category of finitely generated Gorenstein projective $R$-modules is contravariantly finite if $R$ is Gorenstein or G-regular. 
Conversely, Christensen, Piepmeyer, Striuli, and Takahashi \cite{CPST08} prove that if $R$ is a complete local ring and this category is contravariantly finite, then $R$ is either G-regular or Gorenstein.
For artin algebras, combining results of Beligiannis \cite{B98} and Beligiannis and Krause \cite{BK} yields a characterization of strong G-regularity in terms of the homological finiteness of a certain thick subcategory of the category of finitely generated modules. 
We can establish a higher-dimensional analogue of the above characterization for artin algebras:
%More precisely, the strong G-regularity of $R$ is characterized in terms of the homological finiteness of a certain thick subcategory of the category $\cm(R)$ of maximal Cohen--Macaulay $R$-modules. 

\begin{thm}[Theorem \ref{contfin}]\label{main4}
Let $R$ be a Cohen--Macaulay non-Gorenstein complete local ring.
Then $R$ is a strongly G-regular ring if and only if the smallest thick subcategory of $\cm(R)$ containing $R$ and the canonical module $\omega_R$ is covariantly finite, if and only if it is contravariantly finite.
\end{thm}

\noindent
Here, $\cm(R)$ stands for the category of maximal Cohen--Macaulay $R$-modules.
Applying this theorem, under an additional assumption, we recover a theorem of Dey, Kimura, Liu, and Otake \cite{DKLO25}, which states that a non-Gorenstein Cohen--Macaulay local ring of finite CM-representation type is strongly G-regular.

%The theory of algebras of finite representation type has been playing a fundamental role in representation theory since Gabriel \cite{G} established a celebrated classification theorem. 
%A classical result in this theory is the Auslander--Ringel--Tachikawa theorem \cite{A76,RT}, which asserts that an artin algebra $\Lambda$ is of finite representation type if and only if every $\Lambda$-module is a direct sum of finitely generated ones.
%This theorem shows that finite representation type, although formulated in terms of finitely generated modules, determines the structure of the entire module category.

Let $\Lambda$ be an artin algebra.
Beligiannis \cite{B98} calls $\Lambda$ {\em CM-finite} if there exist only finitely many isomorphism classes of indecomposable finitely generated Gorenstein projective $\Lambda$-modules.
Also, Chen \cite{Che12} calls $\Lambda$ {\em CM-free} if every finitely generated Gorenstein projective $\Lambda$-module is projective.
Note that CM-freeness is the same as G-regularity.
Ringel and Zhang \cite{RZ} call $\Lambda$ {\em weakly Gorenstein} if every finitely generated $\Lambda$-module $M$ with $\Ext_\Lambda^{>0}(M,\Lambda)=0$ is Gorenstein projective.

As a Gorenstein projective analogue of the classical Auslander--Ringel--Tachikawa theorem \cite{A76,RT}, Chen \cite{Che08} proves that when $\Lambda$ is Iwanaga--Gorenstein, $\Lambda$ is CM-finite if and only if every Gorenstein projective $\Lambda$-module is a direct sum of finitely generated Gorenstein projective $\Lambda$-modules.
%Thus, it is natural to ask whether an Auslander--Ringel--Tachikawa-type theorem holds for Gorenstein projective modules, that is, whether finitely generated Gorenstein projective modules determine the structure of the category of all Gorenstein projective modules.
Later, Beligiannis \cite{B98} shows that every Gorenstein projective $\Lambda$-module is a direct sum of finitely generated Gorenstein projective $\Lambda$-modules if and only if $\Lambda$ is both CM-finite and {\em virtually Gorenstein}.
Virtual Gorensteinness is a basic assumption in the study of Gorenstein projective modules via cotorsion pairs and approximation theory; see \cite{B05,B98,BR,BK,Che17} for instance.
As every Iwanaga--Gorenstein artin algebra is virtually Gorenstein, Beligiannis's theorem extends Chen's.

Chen \cite[Problems A,B,C]{Che17} formulated the following three problems for an artin algebra $\Lambda$.

\begin{enumerate}[\rm(A)]
\item Is $\Lambda$ CM-finite if and only if every Gorenstein projective $\Lambda$-module is a direct sum of finitely generated Gorenstein projective $\Lambda$-modules?
\item Is $\Lambda$ virtually Gorenstein if it is CM-finite?
\item Is every Gorenstein projective $\Lambda$-module projective if $\Lambda$ is CM-free?
\end{enumerate}

\noindent Beligiannis's theorem \cite{B98} shows that the ``if'' part of (A) holds, while the ``only if'' part is equivalent to (B). 
Thus, (A) and (B) are equivalent, and the affirmativity of either of them implies the affirmativity of (C).
%The aim of this paper is to study these problems using methods in commutative algebra.
We give a negative answer to (C), which hence gives a negative answer to all the three problems:

\begin{thm}[Theorem \ref{hanrei}]\label{main1}
Let $k$ be any field.
Then there exists a commutative local finite-dimensional $k$-algebra $R$ which is weakly Gorenstein and G-regular (i.e., CM-free) but not strongly G-regular, that is to say, the ring $R$ possesses an infinitely generated Gorenstein projective module which is not projective.
\end{thm}

\noindent
%For non-Gorenstein artinian local rings, quasi-dominance is equivalent to strong G-regularity, while a typical example in Theorem \ref{main1} is not dominant \cite{T23}.
%These observations leads us to the proof of Theorem \ref{main1}.
The notion of a strong G-regular ring is also introduced by Atkins and Vraciu \cite{AV}, whose definition is equivalent to being both weakly Gorenstein and G-regular.
As an application of the above theorem, we observe that strong G-regularity in the sense of Atkins and Vraciu is strictly weaker than strong G-regularity in our sense.
%One of the rings $R$ given in the above theorem has already turned out to be G-regular \cite{T23}, we prove which admits a non-projective Gorenstein projective module by constructing a nontrivial totally acyclic complex.

The organization of this paper is as follows.
In Section 2, we develop the theory of quasi-dominant rings and prove the classification of subcategories stated in Theorem \ref{main2}.
In Section 3, we characterize strong G-regularity in terms of derived categories and quasi-dominance, and prove Theorem \ref{main3}.
In Section 4, we relate strong G-regularity to the homological finiteness of thick subcategories, and obtain Theorem \ref{main4}.
In the final Section 5, we provide examples concerning strong G-regularity and prove Theorem \ref{main1}, which answers Chen's questions in the negative.

Throughout this paper, we adopt the following convention.
\begin{conv}
Let $R$ be a commutative noetherian ring. 
For a prime ideal $\p$ of $R$, denote by $\kappa(\p)$ the residue field of $R_\p$, that is, $\kappa(\p)=R_\p/\p R_\p$.
For an ideal $I$ of $R$, the set of prime ideals of $R$ containing $I$ is denoted by $\V(I)$.
All subcategories are assumed to be strictly full.
We denote by $\Mod R$ the category of $R$-modules and by $\mod R$ the subcategory consisting of finitely generated $R$-modules. 
The bounded derived category of $\mod R$ is denoted by $\Db(R)$.
We view $\mod R$ as the subcategory of $\Db(R)$.
\end{conv}

%%%%%%%%%%%%%%%%%%%%%%%%%%%%%%%%%%%%%%%%%
\section{Quasi-dominant rings}

In this section, we shall introduce the notion of quasi-dominance for arbitrary local rings by extending the definition of quasi-dominance for Cohen--Macaulay local rings \cite{T23}, and investigate its fundamental properties.
Quasi-dominance originates from dominance, which has also been introduced in the same paper \cite{T23}.

Recall that a {\em thick} subcategory of a triangulated category $\T$ is defined as a subcategory closed under shifts, cones, and direct summands.
The \textit{thick closure} of a subcategory $\X$ of $\T$ is defined to be the smallest thick subcategory containing $\X$, and denoted $\thick_\T \X$.
We now recall the definitions of dominance and local dominance.
For our purpose, we adopt the equivalent definitions; see \cite[Corollary 10.8 and Remark 10.9]{T23}.

\begin{dfn}
\begin{enumerate}[(1)]
\item
Suppose that $R$ is a local ring with residue field $k$.
We say that $R$ is \textit{dominant} provided that $k$ belongs to $\thick_{\Db(R)}\{R, X\}$ for every object $X\in\Db(R)$ with infinite projective dimension.
\item
We say that $R$ is \textit{locally dominant} if the local ring $R_\p$ is dominant for every prime ideal $\p$ of $R$.
\end{enumerate}
\end{dfn}

Recall that a \textit{dualizing complex} is by definition an object $D\in\Db(R)$ of finite injective dimension such that the homothety morphism $R\to\RHom_R(D,D)$ is an isomorphism.
We define quasi-dominance as follows, allowing in the definition of dominance the use of a dualizing complex to build the residue field.

\begin{dfn}\label{def_quasidom}
\begin{enumerate}[(1)]
\item
Let $R$ be a local ring with residue field $k$ admitting a dualizing complex $D$.
We say that $R$ is {\em quasi-dominant} if $k\in\thick_{\Db(R)}\{R,D,X\}$ for every $X\in\Db(R)$ with infinite projective dimension.
\item
We say that $R$ is {\em locally quasi-dominant} if for each $\p\in\spec R$ the local ring $R_\p$ is quasi-dominant.
\end{enumerate}
\end{dfn}

\begin{rem}\label{4}
\begin{enumerate}[(1)]
\item
When $R$ is a Cohen--Macaulay local ring, $R$ is quasi-dominant in our sense if and only if $R$ is quasi-dominant in the sense of \cite{T23}.
We refer the reader to \cite[Proposition 10.15(2)]{T23}.
\item
Let $R$ be a local ring with residue field $k$ admitting a dualizing complex $D$.
Then $R$ is quasi-dominant if and only if $R$ is either Gorenstein dominant, or non-Gorenstein with $k\in \thick_{\Db(R)}\{R,D\}$.
\end{enumerate}
\end{rem}

Recall that $R$ is said to have {\em finite CM-representation type} if it has only finitely many isomorphism classes of indecomposable maximal Cohen--Macaulay modules. 
Here are examples of quasi-dominant local rings.

\begin{eg}\label{1}
Let $R$ be a local ring with maximal ideal $\m$ and residue field $k$.
\begin{enumerate}[(1)]
\item
If $R$ is dominant, then $R$ is quasi-dominant.
This is immediate from the definitions.
Hence, $R$ is quasi-dominant if $R$ is a hypersurface, or a Cohen--Macaulay ring with minimal multiplicity and infinite residue field, or more generally, a {\em Burch} ring in the sense of \cite{burch}, or if $\m$ is {\em quasi-decomposable} in the sense of \cite{fiber}.
\item
If $R$ is an excellent Cohen--Macaulay local ring of finite CM-representation type, then it is quasi-dominant; see \cite[Proposition 10.15(4)]{T23}.
We will remove the assumption of excellence in Corollary \ref{8}.
\item
If the residue field $k$ is infinite, and $R$ is not Gorenstein but {\em almost Gorenstein} in the sense of \cite{GTT15}, then $R$ is a quasi-dominant local ring.
We refer the reader to \cite[Proposition 10.15(5)]{T23}.
\end{enumerate}
\end{eg}

In the following, we explore the quasi-dominance property of Veronese subrings.
For a group $G$ and a $G$-graded ring $A$, we denote by $\mod^GA$ the category of finitely generated $G$-graded $A$-modules.

\begin{eg}\label{vrns}
In what follows, we use the basic facts on Veronese subrings stated in \cite[Chapter 3]{GW78} and \cite[Exercise 3.6.21]{BH}.
Let $B=k[x_1,\ldots, x_d]$ be a polynomial ring over a field $k$ with $\deg(x_i)=1$ for all $1\le i\le d$.
Let $A$ be the $n$th Veronese subring of $B$.
Then $A$ is Cohen--Macaulay.
Denote by $\omega_A$ the canonical module of $A$.
The assignment $M=\bigoplus_{i\in\ZZ}M_i\mapsto M=\bigoplus_{i=0}^{n-1}M_{\overline{i}}$ where $M_{\overline{i}}=\bigoplus_{j\in\ZZ} M_{nj+i}$ gives an exact functor $\eta:\mod^\ZZ B\to\mod^{\ZZ/n\ZZ}B$; note that $B$ is $\ZZ$-graded and $\ZZ/n\ZZ$-graded.
Then $A=B_{\overline{0}}$ and $\omega_A=B_{\overline{-d}}$.
Let $R$ be the localization of $A$ at its graded maximal ideal.
We investigate if $R$ is either Gorenstein or quasi-dominant, using the characterization \cite[Example 10.8]{GTT15} of almost Gorenstein graded rings.
If $A$ is an almost Gorenstein graded ring, $R$ is an almost Gorenstein local ring, so $R$ is either Gorenstein or quasi-dominant.
\begin{enumerate}[(1)]
\item
Let $n=1$ or $d=1$.
Then $A$ is regular, and so is $R$.
Hence $R$ is both Gorenstein and quasi-dominant.
\item
Let $n=2$.
If $d$ is even, $A$ is Gorenstein and so is $R$.
If $d=3$, then $A$ is not Gorenstein but almost Gorenstein, so $R$ is quasi-dominant.
Now, suppose that $d$ is odd and at least $5$.
Then $A$ is not almost Gorenstein.
However, $R$ is a quasi-dominant local ring.
In fact, as $B$ is regular, $k$ belongs to $\thick_{\mod B}B$.
Since $B=A\oplus\omega_A$, we observe that $k\in\thick_{\Db(A)}\{A, \omega_A\}$, which implies that $k\in\thick_{\Db(R)}\{R,\omega_R\}$.
\item
Let $n=3$.
If $d=2$, then $R$ has minimal multiplicity, and is quasi-dominant if $k$ is infinite by Example \ref{1}(1).
If $d=3$, then $A$ is Gorenstein, and so is $R$.
Let $d=4$.
Then $A$ is not almost Gorenstein.
However, $R$ is quasi-dominant.
Indeed, the minimal free resolution of the graded $B$-module $M=B/(x_1,\dots,x_4)^2$ has the form
$$
0 \to B(-5)^{\oplus 4}\to B(-4)^{\oplus 15}\to B(-3)^{\oplus 20}\to B(-2)^{\oplus 10}\to B\to M\to 0.
$$
Sending this by the exact functor $\eta$ and taking the degree $\overline2$ components, we obtain an exact sequence 
$$
0 \to B_{\overline{0}}^{\oplus 4}\to B_{\overline{1}}^{\oplus 15}\to B_{\overline{2}}^{\oplus 20}\to B_{\overline{0}}^{\oplus 10}\to B_{\overline{2}}\to 0
$$
of $A$-modules.
This shows $B_{\overline{1}}$ is in $\thick_{\mod A}\{B_{\overline0}, B_{\overline{2}}\}=\thick_{\mod A}\{A, \omega_A\}$, and so is $B=B_{\overline{0}}\oplus B_{\overline{1}}\oplus B_{\overline{2}}$.
As $k$ belongs to $\thick_{\mod B}B$, we get $k\in \thick_{\Db(A)}B=\thick_{\Db(A)}\{A, \omega_A\}$.
Thus $k\in\thick_{\Db(R)}\{R,\omega_R\}$.
%\item
%If either $n$ or $d$ is increased beyond the range considered in (3), then, whenever the existing criteria fail to confirm that the ring is dominant or almost Gorenstein, we currently do not know how to determine whether it is quasi-dominant.
\end{enumerate}
\end{eg}

\begin{ques}
Can one determine whether $R$ is either Gorenstein or quasi-dominant for other $n$ and $d$\,?
\end{ques}

Next, we investigate how the quasi-dominance property is transferred by fundamental algebraic operations.
The proposition below shows the stability of quasi-dominance under modding out by a regular element.
Note that the same statement holds for dominance; see \cite[Theorem 5.6]{T23}.

\begin{prop}\label{qdomdeform}
Let $(R,\m, k)$ be a local ring with a dualizing complex $D$.
Let $x\in \m$ be $R$-regular.
Then:\\
{\rm (1)} If $R/xR$ is quasi-dominant, then so is $R$.\quad
{\rm (2)} If $x\notin \m^2$ and $R$ is quasi-dominant, then so is $R/xR$.
\end{prop}

\begin{proof}
When $R$ is Gorenstein, the assertion follows from Remark \ref{4}(2) and \cite[Theorem 5.6]{T23}.
We consider the case where $R$ is not Gorenstein.
Set $D_{R/xR}=\RHom_R(R/xR, D)$, which is a dualizing complex of $R/xR$.

(1) It suffices to show that $k\in\thick_{\Db(R/xR)}\{R/xR, D_{R/xR}\}$ implies $k\in\thick_{\Db(R)}\{R, D\}$; see Remark \ref{4}(2).
The two exact triangles $R \xrightarrow{x} R \to R/xR\rightsquigarrow$ and $D_{R/xR} \to D\xrightarrow{x} D \rightsquigarrow$ in $\Db(R)$ show that $R/xR$ and $D_{R/xR}$ belong to $\thick_{\Db(R)}\{R,D\}$.
Using the natural exact functor $\Db(R/xR)\to\Db(R)$, we are done.

(2) By Remark \ref{4}(2), assuming $k\in\thick_{\Db(R)}\{R, D\}$, we have only show $k\in\thick_{\Db(R/xR)}\{R/xR, D_{R/xR}\}$.
The maximal ideal $\m$ is in $\thick_{\Db(R)}\{R,D\}$. 
The exact functor $(-)\otimes_R^{\mathbf L} R/xR:\Db(R)\to \Db(R/xR)$ yields $\m\otimes_R^{\mathbf L} R/xR\in\thick_{\Db(R/xR)}\{R/xR, D\otimes_R^{\mathbf L} R/xR\}$.
As $x\notin\m^2$ is $R$-regular, we have $\m\otimes_R^{\mathbf L} R/xR\cong\m\otimes_R R/xR\cong k\oplus\m/x R$.
Therefore, $k$ belongs to $\thick_{\Db(R/xR)}\{R/xR, D\otimes_R^{\mathbf L} R/xR\}$.
There are isomorphisms
$$
\begin{array}{l}
\RHom_{R/xR}(D\otimes_R^{\mathbf L}R/xR, D_{R/xR})
=\RHom_{R/xR}(D\otimes_R^{\mathbf L}R/xR,\RHom_R(R/xR, D))\\
\qquad\cong \RHom_R(D\otimes_R^{\mathbf L}R/xR, D)
\cong \RHom_R(R/xR, \RHom_R(D, D))
\cong \RHom_R(R/xR, R) \cong (R/xR)[-1],\\
D\otimes_R^{\mathbf L}R/xR 
\cong \RHom_{R/xR}( \RHom_{R/xR}(D\otimes_R^{\mathbf L}R/xR, D_{R/xR}), D_{R/xR})\\
\qquad\cong \RHom_{R/xR}((R/xR)[-1], D_{R/xR}) \cong  D_{R/xR}\,[1]
\end{array}
$$
in $\Db(R/xR)$.
It is observed that the residue field $k$ belongs to $\thick_{\Db(R/xR)}\{R/xR, D_{R/xR}\}$.
\end{proof}

Similarly as in the proof of \cite[Corollary 5.8]{T23}, the corollary below is deduced from Proposition \ref{qdomdeform}.

\begin{cor}\label{compqdom}
Let $(R,\m)$ be a local ring with a dualizing complex $D$.
Then $R$ is quasi-dominant if and only if so is the formal power series ring $R[\![X]\!]$ over $R$, if and only if so is the ($\m$-adic) completion $\widehat{R}$ of $R$.
\end{cor}

Applying the above result, we can remove the excellence assumption from Example \ref{1}(2).

\begin{cor}\label{8}
Let $R$ be a Cohen--Macaulay local ring with a canonical module.
Suppose that $R$ has finite CM-representation type.
Then $R$ is quasi-dominant.
\end{cor}

\begin{proof}
The completion $\widehat{R}$ is also a Cohen--Macaulay local ring of finite CM-representation type by \cite[Remark 4.9]{DKLO25}.
As $\widehat{R}$ is excellent, Example \ref{1}(2) implies that $\widehat{R}$ is quasi-dominant, and so is $R$ by Corollary \ref{compqdom}.
\end{proof}

One of the main advantages of dominance is that it provides a unified framework for subcategory classification; under suitable assumptions related to dominance, resolving subcategories and thick subcategories can be described in terms of specialization-closed subsets of $\operatorname{Spec} R$.
In what follows, we shall obtain restricted correspondences under weaker hypotheses to apply them to locally quasi-dominant rings. 

Denote by $\c(R)$ the subcategory of $\mod R$ consisting of modules $M$ with $\depth M_\p\ge\depth R_\p$ for all prime ideals $\p$ of $R$.
Let $\M$ be a subcategory of $\mod R$.
We say that a subcategory $\X$ of $\M$ is {\em thick} if it is closed under direct summands and short exact sequences in $\M$, that is, $\X$ satisfies the following conditions.
\begin{itemize}
\item
Let $A$ be an $R$-module and $B$ a direct summand of $A$.
If $A$ belongs to $\X$, then $B$ also belongs to $\X$.
\item
Let $0\to A\to B\to C\to0$ be a short exact sequence of $R$-modules with $A,B,C\in\M$.
If two of the three modules $A,B,C$ belong to $\X$, then the third module also belongs to $\X$.
\end{itemize}
The \textit{thick closure} $\thick_\M \X$ of $\X$ in $\M$ is defined as the smallest thick subcategory of $\M$ containing $\X$.
The following proposition says that thick subcategories containing $R$ bijectively correspond in a natural way.

\begin{prop}\label{2}
There are one-to-one correspondences
$$
\xymatrix@C+2pc{
{\left\{
\begin{matrix}
\text{Thick subcategories}\\
\text{of $\Db(R)$ containing $R$}
\end{matrix}
\right\}}
\ar@<.7mm>[r]^-{()\cap\mod R}
&
{\left\{
\begin{matrix}
\text{Thick subcategories}\\
\text{of $\mod R$ containing $R$}
\end{matrix}
\right\}}
\ar@<.7mm>[l]^-{\thick_{\Db(R)}()}
\ar@<.7mm>[r]^-{()\cap\c(R)}
&
{\left\{
\begin{matrix}
\text{Thick subcategories}\\
\text{of $\c(R)$ containing $R$}
\end{matrix}
\right\}}.
\ar@<.7mm>[l]^-{\thick_{\mod R}()}}
$$
\end{prop}

\begin{proof}
The first pair of mutually inverse bijections is due to Krause and Stevenson \cite[Theorem 1]{KS}; see also \cite[Lemma 10.5]{T23}.
To get the second pair of mutually inverse bijections, fix a thick subcategory $\X$ of $\mod R$ containing $R$, and a thick subcategory $\Y$ of $\c(R)$ containing $R$.
Then, clearly, $\X\cap\c(R)$ is a thick subcategory of $\c(R)$ containing $R$, and $\thick_{\mod R}\Y$ is a thick subcategory of $\mod R$ containing $R$.

It is obvious that $\thick_{\mod R}(\X\cap\c(R))$ is contained in $\X$.
Pick any $X\in\X$.
There exists an integer $n\ge0$ such that the $n$th syzygy $\syz^nX$ belongs to $\c(R)$.
Since $\X$ is thick and contains $R$, every finitely generated projective $R$-module and $\syz^nX$ belong to $\X\cap\c(R)$.
There is a family $\{0\to\syz^{i+1}X\to P_i\to\syz^iX\to0\}_{i=0}^{n-1}$ of short exact sequences in $\mod R$ with each $P_i$ projective, so that $P_i\in\X\cap\c(R)$.
We inductively observe that $X\in\thick_{\mod R}(\X\cap\c(R))$.
We thus obtain the equality $\X=\thick_{\mod R}(\X\cap\c(R))$.

We claim that for every $M\in\thick_{\mod R}\Y$ there exists an integer $r\ge0$ such that $\syz^rM$ is in $\Y$.
Indeed, consider the subcategory $\Z=\{Z\in\mod R\mid\text{$\syz^iZ$ is in $\Y$ for all $i\gg0$}\}$ of $\mod R$.
Then $\Z$ is a thick subcategory of $\mod R$ containing $\Y$.
Hence $\Z$ contains $\thick_{\mod R}\Y$, and we find an integer $r\ge0$ with $\syz^rM\in\Y$.

Clearly, $\Y$ is contained in $(\thick_{\mod R}\Y)\cap\c(R)$.
Pick any $R$-module $M$ in $(\thick_{\mod R}\Y)\cap\c(R)$.
The above claim implies that $\syz^rM$ belongs to $\Y$ for some $r\ge0$.
We have a series $\{0\to\syz^{j+1}M\to Q_j\to\syz^jM\to0\}_{j=0}^{r-1}$ of exact sequences in $\mod R$ with $Q_j$ projective for all $0\le j\le r-1$.
Note that $\syz^{j+1}M,Q_j,\syz^jM$ are in $\c(R)$ and $Q_j$ is in $\Y$.
As $\Y$ is thick in $\c(R)$, we deduce that $M$ is in $\Y$.
We obtain $\Y=(\thick_{\mod R}\Y)\cap\c(R)$.
\end{proof}

We now prepare the terminology needed to state the main result of this section. 
We put:
$$
\begin{array}{l}
\X_\p=\{X_\p\mid X\in\X\}\subseteq\Db(R_\p)\text{ for }\X\subseteq\Db(R)\text{ and }\p\in\spec R,\\
\nf(M)=\{\p\in\spec R\mid\text{$M_\p$ is nonfree over $R_\p$}\}\text{ for }M\in\mod R,\\
\ipd(X)=\{\p\in\spec R\mid\text{$X_\p$ has infinite projective dimension over $R_\p$}\}\text{ for }X\in\Db(R),\\
\ipd(\X)=\bigcup_{X\in\X}\ipd(X)\text{ for }\X\subseteq\Db(R)\text{ and }\nf(\Y)=\bigcup_{Y\in\Y}\nf(Y)\text{ for }\Y\subseteq\mod R,\\
\ipd^{-1}(\Phi)=\{M\in\mod R\mid\ipd(M)\subseteq\Phi\}\text{ and }\nf^{-1}(\Phi)=\{C\in\c(R)\mid\nf(C)\subseteq\Phi\}\text{ for }\Phi\subseteq\spec R,\\
\sing R=\{\p\in\spec R\mid\text{$R_\p$ is not regular\}}\text{ and }\c_0(R)=\{M\in\c(R)\mid\nf(M)\subseteq\Max R\}.
\end{array}
$$
The set $\sing R$ is called the \textit{singular locus} of $R$.
A subset $W$ of $\spec R$ is called \textit{specialization-closed} if $\V(\p)$ is contained in $W$ for all $\p\in W$.
Denote by $\Dsg(R)$ the {\em singularity category} of $R$, that is, the Verdier quotient of $\Db(R)$ by the objects of finite projective dimension.
Let $\pi:\Db(R)\to\Dsg(R)$ be the canonical functor. 
A subcategory of $\mod R$ is called {\em resolving} if it contains $R$ and is closed under direct summands, extensions, and syzygies.
The following theorem is the most general result in this paper concerning subcategory classification.
Here, for subsets $A,B,X$ of $\spec R$, we say that $X$ is {\em between $A$ and $B$} if the inclusions $A\subseteq X\subseteq B$ hold.

\begin{thm}\label{3}
For a subcategory $\E$ of $\Db(R)$ containing $R$, the following two conditions are equivalent.
\begin{enumerate}[\rm(1)]
\item
For every prime ideal $\p$ of $R$, and for every object $X$ of $\Db(R_\p)$ which has infinite projective dimension over $R_\p$, the residue field $\kappa(\p)$ belongs to the thick closure $\thick_{\Db(R_\p)}(\E_\p\cup\{X\})$.
\item
There is a commutative diagram of mutually inverse bijections
$$
\xymatrix@R-1pc@C+2pc{
{\left\{
\begin{matrix}
\text{Resolving subcategories $\X$ of $\mod R$}\\
\text{with $\X\subseteq \c(R)$ and $\E\subseteq\thick_{\Db(R)}\X$}
\end{matrix}
\right\}}
\ar@<.7mm>[r]^-\nf
\ar@{=}[d]
&
{\left\{
\begin{matrix}
\text{Specialization-closed subsets of $\spec R$}\\
\text{between $\ipd(\E)$ and $\sing R$}
\end{matrix}
\right\}}
\ar@<.7mm>[l]^-{\nf^{-1}}
\ar@<.7mm>[d]^-{\ipd^{-1}}
\\
{\left\{
\begin{matrix}
\text{Thick subcategories $\X$ of $\c(R)$}\\
\text{with $\E\subseteq\thick_{\Db(R)}\X$}
\end{matrix}
\right\}}
\ar@<.7mm>[r]^-{\thick_{\mod R}()}
\ar@<.7mm>[d]^-{\thick_{\Dsg(R)}\pi()}
&
{\left\{
\begin{matrix}
\text{Thick subcategories $\X$ of $\mod R$}\\
\text{with $\E\subseteq\thick_{\Db(R)}\X$}
\end{matrix}
\right\}}
\ar@<.7mm>[l]^-{()\cap\c(R)}
\ar@<.7mm>[u]^-\ipd
\ar@<.7mm>[d]^-{\thick_{\Db(R)}()}
\\
{\left\{
\begin{matrix}
\text{Thick subcategories of $\Dsg(R)$}\\
\text{containing $\pi(\E)$}
\end{matrix}
\right\}}
\ar@<.7mm>[u]^-{\pi^{-1}()\cap\c(R)}
\ar@<.7mm>[r]^-{\pi^{-1}}
&
{\left\{
\begin{matrix}
\text{Thick subcategories of $\Db(R)$}\\
\text{containing $\E$}
\end{matrix}
\right\}.}
\ar@<.7mm>[u]^-{()\cap\mod R}
\ar@<.7mm>[l]^-\pi
}
$$
\end{enumerate}
\end{thm}

\begin{proof}
We focus on the two squares in the diagram in (2): the one (call $\U$) consisting of the upper four sets, and the other (call $\L$) consisting of the lower four sets.
In $\L$, the two horizontal one-to-one correspondences and the right vertical one are obtained as restrictions of Proposition \ref{2} and \cite[Lemma 10.5]{T23}.
Applying \cite[Remark 10.2(9)]{T23}, we deduce the commutativity of $\L$ and the left vertical one-to-one correspondence.

Next, we consider the square $\U$.
It is straightforward to verify that the maps $\nf,\ipd,\ipd^{-1}$ in $\U$ are well-defined.
%Note that $\F:=\ipd^{-1}(\ipd(\E))$ is a thick subcategory of $\Db(R)$ containing $\E$. The commutativity of $\L$ yields $\F=\thick_{\Db(R)}(\F\cap\c(R))$, while it is easily seen that $\F\cap\c(R)=\nf^{-1}(\ipd(\E))$. It follows that $\F=\thick_{\Db(R)}(\nf^{-1}(\ipd(\E)))$.
%Using this equality, we easily deduce that the map $\nf^{-1}$ in the square $\U$ is well-defined.
We have the commutative diagram below, where $\theta$ is an inclusion map.
Call this diagram $\VV$.
$$
\xymatrix@R-1pc@C+2pc{
{\left\{
\begin{matrix}
\text{Resolving subcategories $\X$ of $\mod R$}\\
\text{with $\X\subseteq \c(R)$ and $\E\subseteq\thick_{\Db(R)}\X$}
\end{matrix}
\right\}}
\ar[r]^-\nf
&
{\left\{
\begin{matrix}
\text{Specialization-closed subsets of $\spec R$}\\
\text{between $\ipd(\E)$ and $\sing R$}
\end{matrix}
\right\}}
\\
{\left\{
\begin{matrix}
\text{Thick subcategories $\X$ of $\c(R)$}\\
\text{with $\E\subseteq\thick_{\Db(R)}\X$}
\end{matrix}
\right\}}
\ar@<1.8mm>[r]^-{\thick_{\mod R}()}_-{\cong}
\ar@{^{(}->}[u]_-\theta
&
{\left\{
\begin{matrix}
\text{Thick subcategories $\X$ of $\mod R$}\\
\text{with $\E\subseteq\thick_{\Db(R)}\X$}
\end{matrix}
\right\}.}
\ar[u]^-\ipd\ar@<1.8mm>[l]^-{()\cap\c(R)}
}
$$
Note that $(\theta\circ()\cap\c(R)\circ\ipd^{-1})(\Phi)=\nf^{-1}(\Phi)$ for each subset $\Phi$ of $\spec R$.
Thus the map $\nf^{-1}$ in the square $\U$ is well-defined.
For a specialization-closed subset $W$ of $\spec R$ contained in $\sing R$, it follows from \cite[(b) and (g) in the proof of Theorem 10.10]{T23} that $\nf(\nf^{-1}(W))=W=\ipd(\ipd^{-1}(W))$.

Now, suppose that (1) is satisfied.
To observe that (2) holds, it suffices to prove that $\X=\nf^{-1}(\nf(\X))$ for any resolving subcategory $\X$ of $\mod R$ with $\X\subseteq \c(R)$ and $\E\subseteq\thick_{\Db(R)}\X$. 
Indeed, suppose that this statement is shown.
Then the maps $\nf,\nf^{-1}$ in $\U$ are mutually inverse bijections.
The commutativity of $\VV$ implies that the map $\ipd$ in $\VV$ is injective, which implies that the maps $\ipd,\ipd^{-1}$ in $\U$ are mutually inverse bijections.
The inclusion map $\theta$ in $\VV$ turns out to be an identity map, and assertion (2) follows.

Clearly, $\X$ is contained in $\nf^{-1}(\nf(\X))$. 
Let $M$ be an $R$-module in $\nf^{-1}(\nf(\X))$.
Similarly as in \cite[Remark 10.13]{T23}, we may assume that $(R,\m,k)$ is local.
We may also assume that $M$ is nonfree over $R$.
Fix $\p\in\nf(\X)$.
As $\X$ is contained in $\c(R)$, there exists an $R$-module $Z\in\X$ such that $Z_\p$ has infinite projective dimension over $R_\p$.
By our assumption that (1) holds, we have $\kappa(\p)\in\thick_{\Db(R_\p)}(\E_\p\cup\{Z_\p\})\subseteq\thick_{\Db(R_\p)}\X_\p$.
Choose an $R$-module $Y\in\X$ with $\kappa(\p)\in\thick_{\Db(R_\p)}Y_\p$.
It follows from \cite[(4)$\Rightarrow$(1) in Lemma 10.6]{T23} and \cite[Lemma 3.2(1)]{crspd} that $\syz^{\depth R_\p}\kappa(\p)$ belongs to the additive closure $\add(\X_\p)$.
In view of \cite[Remark 3.3(8)]{T23}, we have that $\c_0(R_\p)$ is contained in $\add(\X_\p)$ for all $\p\in\nf(\X)$.
Letting $\p=\m$, we get $\c_0(R)\subseteq\X$.
Apply \cite[Theorem 3.8]{T23} to obtain a short exact sequence $0\to C\to M\oplus N\to X\to0$ in $\mod R$ with $C\in\c_0(R)$ and $X\in\X$.
Consequently, $M$ is in $\X$.
We conclude that the equality $\X=\nf^{-1}(\nf(\X))$ holds.

Finally, we show that (2) implies (1).
Fix a prime ideal $\p$ of $R$ and an object $X\in\Db(R_\p)$ of infinite projective dimension.
Choose an $R$-module $M$ with $\thick_{\Db(R_\p)}\{R_\p, X\}=\thick_{\Db(R_\p)}\{R_\p, M_\p\}$, which implies $\p\in\ipd(M)$.
Using the mutually inverse bijections $\ipd$ and $\ipd^{-1}$ in the diagram in (2), we get
$$
R/\p\in \ipd^{-1}(\ipd(M))\subseteq \thick_{\Db(R)}(\E\cup\{M\})\cap \mod R.
$$
Hence $\kappa(\p)$ belongs to $\thick_{\Db(R_\p)}(\E_\p\cup\{M_\p\})$, which coincides with $\thick_{\Db(R_\p)}(\E_\p\cup\{X\})$, since $R\in\E$.
\end{proof}

\begin{rem}
The proof of Theorem \ref{3} says that, for the maps $\nf,\nf^{-1},\ipd,\ipd^{-1}$ which appear in the diagram in the theorem, the following two statements hold true, where $\id$ denotes the identity map.
\begin{enumerate}[(1)]
\item
One has the equalities $\nf\circ\nf^{-1}=\id$ and $\ipd \circ \ipd^{-1}=\id$.
\item
Theorem \ref{3}(2) is equivalent to the equality $\nf^{-1}\circ\nf=\id$, and to the equality $\ipd^{-1}\circ\ipd=\id$.
\end{enumerate}
%Every $\mathcal{C}(R)$ appearing in the diagram of Theorem \ref{3} can be replaced by an arbitrary dominant resolving subcategory in the sense of \cite[Definition 4.4]{crspd}.
\end{rem}

Taking $\E=\{R, D\}$ in Theorem \ref{3}, we obtain a classification of subcategories over locally quasi-dominant rings, which (partially) generalizes \cite[Corollary 10.17]{T23} to the case where $R$ is not Cohen--Macaulay.
The \textit{non-Gorenstein locus} $\ng R$ of $R$ is defined as the set of prime ideals $\p$ of $R$ such that $R_\p$ is not Gorenstein.

\begin{cor}\label{lqdom}
Suppose that $R$ admits a dualizing complex $D$.
Then the ring $R$ is locally quasi-dominant if and only if one has the following mutually inverse bijections.
$$
\xymatrix@R-1pc@C+2pc{
{\left\{
\begin{matrix}
\text{Resolving subcategories $\X$ of $\mod R$}\\
\text{with $\X\subseteq \c(R)$ and $D\in\thick_{\Db(R)}\X$}
\end{matrix}
\right\}}
\ar@<.7mm>[r]^-\nf
\ar@{=}[d]
&
{\left\{
\begin{matrix}
\text{Specialization-closed subsets of $\spec R$}\\
\text{between $\ng R$ and $\sing R$}
\end{matrix}
\right\}}
\ar@<.7mm>[l]^-{\nf^{-1}}
\ar@<.7mm>[d]^-{\ipd^{-1}}
\\
{\left\{
\begin{matrix}
\text{Thick subcategories $\X$ of $\c(R)$}\\
\text{with $R, D\in\thick_{\Db(R)}\X$}
\end{matrix}
\right\}}
\ar@<.7mm>[r]^{\thick_{\mod R}()}
\ar@<.7mm>[d]^{\thick_{\Dsg(R)}\pi()}
&
{\left\{
\begin{matrix}
\text{Thick subcategories $\X$ of $\mod R$}\\
\text{with $R, D\in\thick_{\Db(R)}\X$}
\end{matrix}
\right\}}
\ar@<.7mm>[l]^-{()\cap\c(R)}
\ar@<.7mm>[u]^-\ipd
\ar@<.7mm>[d]^{\thick_{\Db(R)}()}
\\
{\left\{
\begin{matrix}
\text{Thick subcategories of $\Dsg(R)$}\\
\text{containing $\pi(D)$}
\end{matrix}
\right\}}
\ar@<.7mm>[u]^-{\pi^{-1}()\cap\c(R)}
\ar@<.7mm>[r]^-{\pi^{-1}}
&
{\left\{
\begin{matrix}
\text{Thick subcategories of $\Db(R)$}\\
\text{containing $R$ and $D$}
\end{matrix}
\right\}.}
\ar@<.7mm>[u]^-{()\cap\mod R}
\ar@<.7mm>[l]^-\pi
}
$$
\end{cor}

Recall that a local ring $R$ is said to have an {\em isolated singularity} if the localization $R_\p$ is a regular local ring for all nonmaximal prime ideals $\p$ of $R$.
We close the section by stating a remark on the above corollary.

\begin{rem}
Beyond providing a classification theorem, Corollary \ref{lqdom} offers insight into the nature of quasi-dominance.
Consider the case where $\ng R=\sing R$, e.g., where $R$ is a non-Gorenstein local ring with an isolated singularity.
Then quasi-dominance is characterized by the triviality of the subcategories generated by $R$ and $D$.
This observation shows that quasi-dominance is closely related to {\em strong G-regularity}, which we shall introduce in the next section.
Nontrivial examples of locally quasi-dominant rings for which the six sets in the diagram are not singleton sets are known in the Gorenstein case: singular locally dominant rings provide such examples.
On the other hand, such a non-Gorenstein example will be given in Example \ref{not deform sgr}(1).
\end{rem}

%%%%%%%%%%%%%%%%%%%%%%%%%%%%%%%%%%%%%%%
\section{Strongly G-regular rings}

This section introduces the notion of strongly G-regular rings, extending that of G-regular rings in the sense of \cite{T08}.
We shall figure out properties of strongly G-regular rings and consider the questions of Chen, Ringel and Zhang for commutative rings.
We start by recalling the definition of Gorenstein projective modules.

\begin{dfn}
\begin{enumerate}[(1)]
\item
Denote by $\Proj R$ the subcategory of $\Mod R$ consisting of projective $R$-modules, and by $\proj R=\Proj R\cap\mod R$ the subcategory of $\mod R$ consisting of finitely generated projective $R$-modules.
\item
A {\em totally acyclic} $R$-complex is by definition an exact complex over $\Proj R$ whose $P$-dual is again exact for every $P\in\Proj R$.
We denote by $\K(\Proj R)$ the homotopy category of complexes over $\Proj R$, and by $\Ktac (\Proj R)$ the subcategory of $\K(\Proj R)$ consisting of totally acyclic complexes.
\item
An $R$-module $M$ is called \textit{Gorenstein projective} if there exists a totally acyclic complex $P$ such that $M$ is isomorphic to the image of a differential map in $P$.
Denote by $\GProj R$ the subcategory of $\Mod R$ consisting of Gorenstein projective modules.
Let $\Gproj R=\GProj R\cap\mod R$ be the subcategory of $\mod R$ consisting of finitely generated Gorenstein projective modules, which are also called {\em totally reflexive}.
\end{enumerate}
\end{dfn}

We recall the definitions of G-regularity and the condition $\TR$, and then we shall define strong G-regularity.

\begin{dfn}
There are inclusions of subcategories
$$
\Mod R\supseteq{}^\perp\!R\supseteq\GProj R\supseteq\Proj R,\qquad
\mod R\supseteq{}^\perp\!R\cap\mod R\supseteq\Gproj R\supseteq\proj R
$$
where ${}^\perp\!R$ stands for the subcategory of $\Mod R$ consisting of $R$-modules $M$ with $\Ext_R^{>0}(M,R)=0$.
Following \cite{KLOT}, we say that $R$ {\em satisfies $\TR$} if the equality ${}^\perp\!R\cap\mod R=\Gproj R$ holds.
Following \cite{T08}, we say that $R$ is \textit{G-regular} if the equality $\Gproj R=\proj R$ is satisfied.
We thus call $R$ \textit{strongly G-regular} if $\GProj R=\Proj R$.
\end{dfn}

\begin{rem}\label{5}
\begin{enumerate}[(1)]
\item
When $R$ is strongly G-regular, $R$ is G-regular.
When $R$ is Gorenstein, $R$ is G-regular if and only if $R$ is strongly G-regular, if and only if $R$ is regular.
We refer the reader to \cite[Proposition 1.8(1)]{T08}.
\item
Atkins and Vraciu \cite{AV} call $R$ {\em strongly G-regular} if $R$ is an artinian local ring with residue field $k$ and every finitely generated $R$-module $M$ with $\Tor_{>0}^R(M,\EE_R(k))=0$ is free.
Hence an artinian local ring $R$ is strongly G-regular in their sense if and only if $R$ is G-regular and satisfies $\TR$.
\end{enumerate}
\end{rem}

G-regularity and strong G-regularity can be defined for a noncommutative ring.
Indeed, a noncommutative G-regular ring is called \textit{CM-free} and widely studied; see \cite{B98,Che12,Che17,CY,HS,KP}.
Chen asks in \cite[Problem C]{Che17} whether every G-regular artin algebra is strongly G-regular.
On the other hand, in the case of artin algebras, the condition $\TR$ is the same as the {\em weakly Gorenstein} property introduced by Ringel and Zhang \cite{RZ}.
They ask in \cite[Question 9.2]{RZ} if every G-regular artin algebra satisfies $\TR$.
We formulate these two questions due to Chen and to Ringel and Zhang for commutative rings:

\begin{ques}\label{main ques}
Suppose that $R$ is G-regular.\quad
(1) Is $R$ strongly G-regular?\quad
(2) Does $R$ satisfy $\TR$?
\end{ques}

\begin{rem}
\begin{enumerate}[(1)]
\item
Chen's question \cite[Problem C]{Che17} is closely related to his other two questions \cite[Problems A and B]{Che17}. 
Motivated by this, Dey, Kimura, Liu, and Otake \cite{DKLO25} formulate three corresponding questions for a Cohen--Macaulay local ring $R$ which has finite CM-representation type, and answer them affirmatively, assuming that $R$ is complete for the latter two questions.
In particular, Question \ref{main ques} has an affirmative answer if $R$ is a Cohen--Macaulay local ring having finite CM-representation type.
\item
In view of Remark \ref{5}(1), to consider Question \ref{main ques} we should assume that $R$ is non-Gorenstein.
Beligiannis \cite{B98} establishes a relationship between virtual Gorensteinness and strong G-regularity.
These properties are equivalent for non-Gorenstein commutative artinian rings; see \cite[Corollary 5.12]{DKLO25}.
\end{enumerate}
\end{rem}

Next we recall the definitions of annihilators for $R$-linear categories, and Rouquier dimension introduced in \cite{R08}.
Also, we introduce a certain set of prime ideals of $R$ and a certain condition on $R$.

\begin{dfn}
\begin{enumerate}[(1)]
\item
Let $\C$ be an $R$-linear category.
For each object $X\in\C$, denote by $\ann_R X$ the {\em annihilator} of $X$, that is to say, the set of elements $a\in R$ such that $a\Hom_\C(X,X)=0$.
Put $\ann_R \C=\bigcap_{X\in\C} \ann_R X$.
\item
The {\em (Rouquier) dimension} $\dim \T$ of a triangulated category $\T$ is defined as the infimum of integers $n$ such that there exists an object $X\in \T$ such that every object of $\T$ can be obtained from $X$ by taking finite direct sums, direct summands, shifts, and at most $n$ cones; we refer the reader to \cite{R08} for the details.
\item
The \textit{non-strongly G-regular locus} $\nsgr R$ is the set of prime ideals $\p$ such that $R_\p$ is not strongly G-regular.
\item
In what follows, the following condition will become a standard assumption.
\begin{enumerate}
\item[$(\#)$]
The ring $R$ admits a dualizing complex $D$, and $\Db(R)/\thick\{R, D\}$ has finite dimension.
\end{enumerate}
\end{enumerate}	
\end{dfn}

\begin{rem}\label{9}
\begin{enumerate}[(1)]
\item
The condition $(\#)$ is stable under localization by a multiplicatively closed subset of $R$.
\item
By Aoki's work \cite{A21} the latter condition of $(\#)$ holds if $R$ is a quasi-excellent ring of finite Krull dimension.
When $R$ has a dualizing complex, it is universally catenary and has finite Krull dimension; see \cite[Chapter V, \S10]{H}.
The ring $R$ satisfies $(\#)$ as long as $R$ is an excellent ring that admits a dualizing complex.
\end{enumerate}
\end{rem}

The following theorem is the main result of this section.
It especially says that under the assumption of $(\#)$, the non-strongly G-regular locus of $R$ is Zariski-closed, so that strong G-regularity is a local property. 
The equivalence between (1) and (3) also appears in \cite[Theorem 1.3]{DLL26} for artin algebras.

\begin{thm}\label{str G-req quasi-dom}
Suppose a commutative noetherian ring $R$ satisfies the condition $(\#)$.
Then the equality
\begin{equation}\label{6}
\nsgr R=\V(\ann_R (\Db(R)/\thick\{R, D\}))
\end{equation}
of sets of prime ideals of $R$ holds true, and the following seven conditions are equivalent.
\begin{enumerate}[\rm(1)]
\item
The commutative noetherian ring $R$ is strongly G-regular.
\item
The homotopy category $\Ktac (\Proj R)$ of totally acyclic complexes is zero.
\item
The thick closure $\thick_{\Db(R)}\{R, D\}$ coincides with the whole category $\Db(R)$.
\item
For every prime ideal $\p$ of $R$, the localization $R_\p$ is a strongly G-regular local ring.
\item
For every prime ideal $\p$ of $R$, the residue field $\kappa(\p)$ of the local ring $R_\p$ belongs to $\thick_{\Db(R_\p)} \{R_\p, D_\p\}$.
\item
One has $\sing R=\ng R$, and for any $\p\in\ng R$, the residue field $\kappa(\p)$ belongs to $\thick_{\Db(R_\p)} \{R_\p, D_\p\}$.
\item
One has $\sing R=\ng R$, and the localization $R_\p$ is a quasi-dominant local ring for each $\p\in\ng R$.
\end{enumerate}
\end{thm}

\begin{proof}
(1)$\Leftrightarrow$(2):
This is an immediate consequence of the well-known fact that  there exists a triangle equivalence between $\Ktac (\Proj R)$ and the stable category of $\GProj R$; see \cite[Example 3.10]{CET20} for instance.
%Here we include a proof for the convenience of the reader. First, in order to prove the ``only if'' part, we assume that $R$ is strongly G-regular. For a totally acyclic complex $F=(\cdots \to F_{i+1} \xrightarrow{d_{i+1}} F_{i} \xrightarrow{d_i} F_{i-1} \to \cdots)$, we divide it into short exact sequences $\sigma_i: 0\to G_{i+1}\xrightarrow{\iota_{i+1}} F_{i} \xrightarrow{\pi_i} G_{i}\to 0$ for each $i\in\mathbb{Z}$, that is, $G_i=\Im d_i$ and $d_i=\iota_i \circ\pi_i$. For any $i\in\mathbb{Z}$, $\sigma_i$ is split since $G_i\in\GProj R=\Proj R$. There is $\alpha_{i+1}: F_i\to G_{i+1}$ and $\beta_i: G_i\to F_i$ such that $$\alpha_{i+1} \circ\iota_{i+1}=\id_{G_{i+1}}, \pi_i\circ\beta_i=\id_{G_i}, \textrm{ and } \beta_i\circ\pi_i+\iota_{i+1}\circ\alpha_{i+1} =\id_{F_i}.$$ We put $h_i=\beta_i \circ\alpha_i$ for each $i\in\mathbb{Z}$. We obtain $$h_i\circ d_i+d_{i+1}\circ h_{i+1}=(\beta_i\circ \alpha_i)\circ (\iota_i\circ \pi_i)+(\iota_{i+1}\circ \pi_{i+1})\circ (\beta_{i+1}\circ \alpha_{i+1})=\beta_i\circ \pi_i+\iota_{i+1} \circ\alpha_{i+1}=\id_{F_i},$$ which means $F=0$ in $\Ktac (\Proj R)$. Conversely, in general,  if a complex $(\cdots \to  F_{i} \xrightarrow{d_i} F_{i-1} \to \cdots)$ is zero in the homotopy category, then $\Im d_i$ is isomorphic to a direct summand of $F_i$ for any $i\in\mathbb{Z}$. By this, we easily obtain the ``if'' part.

(2)$\Leftrightarrow$(3):
Applying \cite[(1) and (2) of Theorem 5.3]{IK06}, we immediately get this equivalence.

\eqref{6}:
By the equivalence (1)$\Leftrightarrow$(3) which has already been shown, the set $\nsgr R$ consists of the prime ideals $\p$ with $\Db(R_\p) \ne \thick\{R_\p, D_\p\}$.
This equals $\V(\ann_R (\Db(R)/\thick_{\Db(R)}\{R, D\}))$ by \cite[Corollary 3.3]{DM24}.

(3)$\Leftrightarrow$(4):
This equivalence is a direct consequence of the equality \eqref{6} which has already been proved.

(4)$\Rightarrow$(5):
Applying the already shown implication (1)$\Rightarrow$(3) to the localization $R_\p$ yields this implication.

(5)$\Rightarrow$(6):
If $R_\p$ is Gorenstein, then $D_\p$ is isomorphic to a shift of the stalk complex $R_\p$, and the containment $\kappa(\p)\in\thick_{\Db(R_\p)}\{R_\p,D_\p\}$ implies that $R_\p$ is regular.
The equality $\sing R=\ng R$ thus follows.

(6)$\Leftrightarrow$(7):
It is straightforward to deduce this equivalence by applying Remark \ref{4}(2). 

(7)$\Rightarrow$(3):
The ring $R$ is locally quasi-dominant, and Corollary \ref{lqdom} yields $\Db(R)=\thick_{\Db(R)}\{R, D\}$.
\end{proof}

\begin{rem}\label{12}
It is observed from the proof that the equivalences (1)$\Leftrightarrow$(2)$\Leftrightarrow$(3) in Theorem \ref{str G-req quasi-dom} hold for a commutative noetherian ring $R$ having a dualizing complex (we do not need to assume that $R$ satisfies $(\#)$).
\end{rem}

As an application of the above theorem, we have a partial answer to Question \ref{main ques}(2).

\begin{cor}\label{stgregtr}
Suppose that $R$ admits a dualizing complex $D$.
Then the following statements hold.
\begin{enumerate}[\rm(1)]
\item
If $R$ is a strongly G-regular ring, then $R$ satisfies the condition $\TR$.
\item
If Question \ref{main ques}(1) has an affirmative answer for the ring $R$, then so does Question \ref{main ques}(2).
\end{enumerate}
\end{cor}

\begin{proof}
The second assertion follows from the first.
To prove the first assertion, let $M$ be a finitely generated $R$-module such that $\Ext_R^{>0}(M,R)=0$.
Then $\Db(R)=\thick_{\Db(R)} \{R, D\}$ by the equivalence (1)$\Leftrightarrow$(3) in Theorem \ref{str G-req quasi-dom} (see Remark \ref{12}).
As $\RHom_R(M,R)$, $\RHom_R(M,D)$ are homologically bounded, so is $\RHom_R(M,X)$ for all objects $X\in\Db(R)$.
In particular, for every prime ideal $\p$ of $R$ the complex $\RHom_{R_\p}(M_\p, \kappa(\p))=\RHom_R(M,R/\p)_\p$ is homologically bounded, which means that $M_\p$ has finite projective dimension over $R_\p$.
Hence $M$ has finite projective dimension over $R$, and is projective over $R$ since $\Ext_R^{>0}(M,R)=0$.
\end{proof}

Theorem \ref{str G-req quasi-dom} makes it possible to compare strong G-regularity with other properties of commutative rings, as follows.
The third assertion is the same as \cite[Proposition 5.15]{DKLO25}.

\begin{cor}\label{sgrprprng}
Assume that $R$ satisfies the condition $(\#)$.
Then the following statements hold true.
\begin{enumerate}[\rm(1)]
\item
The ring $R$ is strongly G-regular if and only if $R$ is locally quasi-dominant with $\sing R=\ng R$.
\item
If $R$ is a non-Gorenstein quasi-dominant local ring with an isolated singularity, $R$ is strongly G-regular.
\item
If $R$ is Cohen--Macaulay non-Gorenstein local with finite CM-representation type, $R$ is strongly G-regular.
\end{enumerate}
\end{cor}

\begin{proof}
The first assertion follows from the equivalence (1)$\Leftrightarrow$(7) in Theorem \ref{str G-req quasi-dom}, while the second assertion is a direct consequence of the first.
The third assertion follows by the second and Corollary \ref{8}.
\end{proof}

\begin{eg}\label{str G-reg rem2}
In general, a Veronese subring of a polynomial ring over a field has an isolated singularity.
Hence, such a quasi-dominant local ring $R$ as in Example \ref{vrns} is strongly G-regular by Corollary \ref{sgrprprng}(2).
\end{eg}

Taking into account the fact that G-regularity is not necessarily preserved under localization, we see from the theorem stated above that strong G-regularity and G-regularity are distinct notions, and get the following.

\begin{cor}\label{11}
Question \ref{main ques}(1) has a negative answer in general.
\end{cor}

\begin{proof}
Let $R=k\llbracket x,y,z\rrbracket/(x^2, xz, yz)$ be a quotient of a formal power series ring over a field $k$.
As is explained in \cite[Example 6.2]{T08}, the ring $R$ is G-regular, but $R_\p\cong k\llbracket x,y\rrbracket_{(x)}/(x^2)$ is not G-regular for $\p=(x,z)$.
Hence $R_\p$ is not strongly G-regular.
Remark \ref{9}(2) implies that $R$ satisfies $(\#)$.
It follows from Theorem \ref{str G-req quasi-dom} that the ring $R$ is not strongly G-regular.
(Also, as $\p$ is in $\sing R$ but not in $\ng R$, it holds that $\sing R\ne\ng R$.)
\end{proof}

\begin{rem}
\begin{enumerate}[(1)]
\item
The proof of Corollary \ref{11} also says that the {\em non-G-regular locus} of $R$ is not necessarily Zariski-closed even when $R$ is a complete equicharacteristic local ring.
\item
Since the ring $R$ in the proof of Corollary \ref{11} is one-dimensional, this proof does not provide a negative answer to the original Chen's question \cite[Problem C]{Che17} regarding artin algebras.
%In fact, the authors wonder if Question \ref{main ques}(1) should be affirmative when $R$ is artinian.
\end{enumerate}
\end{rem}

The difference between G-regularity and strong G-regularity can also be observed from another point of view.
With respect to modding out by a regular element, G-regularity behaves stably; see \cite[Propositions 4.2 and 4.6]{T08}. However, the following example shows that strong G-regularity does not enjoy such stability.

\begin{eg}\label{not deform sgr} Let $k$ be a field.
In the two items below, denote by $\m$ the maximal ideal of the local ring $R$.
\begin{enumerate}[(1)]
\item
Let $R$ be the completion of the third Veronese subring of the polynomial ring $k[s,t]$, that is to say, 
$$
R=k\llbracket s^3, s^2t, st^2, t^3\rrbracket\cong k\llbracket x,y,z,w\rrbracket/I_2\left(\begin{smallmatrix}
   x & y & z\\
   y & z & w
\end{smallmatrix}\right)
= k\llbracket x,y,z,w\rrbracket/(xz-y^2, xw-yz, yw-z^2).
$$
Then $R$ is a non-Gorenstein dominant local ring with an isolated singularity; see Examples \ref{vrns} and \ref{str G-reg rem2}.
By Corollary \ref{sgrprprng}, $R$ is strongly G-regular.
By Theorem \ref{str G-req quasi-dom} the localization $R_\p$ at $\p=(x,y,z)$ is strongly G-regular.
The element $x\in\m\setminus\m^2$ is regular on $R$.
By \cite[Proposition 4.6]{T08}, $S=R/xR$ is G-regular.
As $w\notin \p$ and $xw=yz$, we get $x\in \p^2 R_\p\cap R=\p^{(2)}$.
Since $S_\p=R_\p/xR_\p$ is not G-regular by \cite[Proposition 4.6]{T08} again, it is not strongly G-regular.
Using Theorem \ref{str G-req quasi-dom} again shows that $S$ is not strongly G-regular.
\item
Let $R=k\llbracket x,y,z\rrbracket/(x^2, xz, yz)$ be the ring in the proof of Corollary \ref{11}, which is not strongly G-regular.
The quotient $S=R/(y-z)R\cong k\llbracket x,y\rrbracket/(x^2, xy, y^2)$ by the regular element $y-z\in\m\setminus\m^2$ is an artinian local ring with minimal multiplicity.
By \cite[Proposition 5.10]{T23}, $S$ is dominant, whence strongly G-regular. 
\end{enumerate}
\end{eg}

\begin{rem}
Let $(R,\m)$ be a local ring and $x\in\m$ regular.
Let $M$ be an $R$-module with $M/xM$ projective over $R/xR$.
If $M$ is finitely generated, it is projective over $R$.
This elementary fact is no longer true even if $M$ is Gorenstein projective.
Indeed, in the situation of Example \ref{not deform sgr}(2), as $R$ is not strongly G-regular, there is a non-projective Gorenstein projective $R$-module $M$.
The element $f=y-z$ is regular, and $M/fM$ is Gorenstein projective $S$-module.
Since $S$ is strongly G-regular, $M/fM$ is a projective module over $S=R/fR$.
\end{rem}

The reason for the failure of deformation in Example \ref{not deform sgr} comes from non-preservation of equality between the non-Gorenstein and singular loci.
One can add an assumption that eliminates this obstruction as follows.

\begin{prop}\label{sgr reg}
Let $R$ be with a dualizing complex $D$, and let $x\in R$ be an $R$-regular element contained in every prime ideal belonging to the singular locus of $R$. If $R/xR$ is strongly G-regular, then so is $R$.
\end{prop}

\begin{proof}
In view of Theorem \ref{str G-req quasi-dom}, it suffices to show $\Db(R)=\thick_{\Db(R)}\{R, D\}$ (see Remark \ref{12}).
By \cite[Main Theorem]{Sc03} (see also \cite[Corollary 1.3]{thgen}), it is enough to prove that $R/\p$ belongs to $\thick_{\Db(R)}\{R, D\}$ for every $\p\in\sing R$. 
The complex $D_{R/xR}:=\RHom_R(R/xR, D)$ is a dualizing complex of $R/xR$.
As $x$ is in $\p$ and $R/xR$ is strongly G-regular, we have  $R/\p=(R/xR)/(\p/xR)\in\thick_{\Db(R/xR)}\{ R/xR, D_{R/xR}\}$.
Since $R/xR$ and $D_{R/xR}$ are in $\thick_{\Db(R)}\{R, D\}$ as we saw in the proof of Proposition \ref{qdomdeform}, we get $R/\p\in \thick_{\Db(R)}\{R, D\}$.
\end{proof}

The first assertion of the theorem below deals with descent of G-regularity along a homomorphism of finite flat dimension, and the second assertion is an application.
For an $R$-module $M$, denote by $\pd_R M$, $\fd_R M$, and $\cx_R M$ the projective dimension, the flat dimension, and the complexity of $M$, respectively, and put $M^\ast=\Hom_R(M,R)$.
When $R$ is local, $\codim R$ stands for the codimension of $R$, i.e., $\codim R=\edim R-\dim R$.

\begin{thm}\label{13}
Let $\phi:R \to S$ be a homomorphism of commutative noetherian rings with $S$ being G-regular.
\begin{enumerate}[\rm(1)]
\item
If $\fd_R S < \infty$ and the image of the induced map $\spec S \to \spec R$ contains $\Max R$, then $R$ is G-regular.
\item
Suppose that $\phi$ is finite and local, and that $R$ is a complete intersection.
Then $\cx_R S = \codim R$.
\end{enumerate}
\end{thm}

\begin{proof}
(1) Let $M \in \Gproj R$.
As $\fd_R S < \infty$, we have $\Tor_{>0}^R(M\oplus M^*,S)=0$ by \cite[Theorem 5.3.6]{Chr}.
We get
$$
\RHom_S(M\otimes_R S,S)\cong \RHom_S(M\otimes_R^\mathbf{L} S,S)\cong \RHom_R(M,S)\cong \RHom_R(M,R)\otimes_R^\mathbf{L}S\cong M^*\otimes_R S,
$$
where the third isomorphism follows from \cite[A.4.23]{Chr}.
%\cite[Theorem 8.4.12(b)]{CFH24}
Also, $\RHom_S(M^*\otimes_R S,S)\cong M^{**}\otimes_R S\cong M\otimes_R S$.
Hence $M\otimes_R S \in \Gproj S=\proj S$ as $S$ is G-regular. For any $\m \in \Max R$, there exists $P \in \spec S$ such that $P\cap R=\m$ by assumption. Since $R/\m$ is a field, we have $S/P\cong(R/\m)^{\oplus \Lambda}$ for some set $\Lambda$, which implies 
$$
(M\otimes_R^\mathbf{L} R/\m)^{\oplus \Lambda}
\cong (M\otimes_R^\mathbf{L} R/\m) \otimes_{R/\m}^\mathbf{L} S/P
\cong (M\otimes_R^\mathbf{L} S)\otimes_S^\mathbf{L} S/P 
\cong (M\otimes_R S)\otimes_S S/P.
$$
It follows that $\Tor^R_{>0}(M, R/\m)=0$, and therefore $M_\m$ is free over $R_\m$.
Thus $M$ is a projective $R$-module.

(2) Let $\m$ and $\n$ be the maximal ideals of $R$ and $S$, respectively.
As $\phi$ is finite and local, $\m S$ is $\n$-primary and $\n\cap R=\m$.
If $\cx_R S=0$, then $\pd_R S<\infty$, $R$ is G-regular by (1), $R$ is regular as it is Gorenstein, and $\cx_R S=\codim R=0$.
Let $\cx_R S>0$.
Let $\widehat{R}$ be the $\m$-adic completion of $R$.
Note that $\widehat{S}:= S\otimes_R \widehat{R}$ coincides with the $\n$-adic completion of $S$.
It is seen from \cite[Corollary 4.7]{T08} that $\widehat S$ is G-regular.
We have $\cx_{\widehat R}\widehat{S}=\cx_R S$ and $\codim \widehat{R}=\codim R$.
Replacing $R$ by $\widehat{R}$, we may assume that the local ring $R$ is ($\m$-adically) complete.

Assume that $\cx_R S < \codim R$.
By \cite[Proposition 2.7]{CDT14}, there exists a finitely generated $R$-module $M$ such that $\pd_R M=\infty$ and $\Tor_{\gg0}^R(M,S)=0$.
We may assume $\cx_R M=1$ by (the proof of) \cite[Proposition 2.2(i)]{Ber}.
Replacing $M$ by its sufficiently high syzygy, we may assume $\Tor_{>0}^R(M,S)=0$, $M$ is maximal Cohen--Macaulay and $2$-periodic by \cite[Theorem 4.1]{E80} (this can be applied as $R$ is a complete local complete intersection).

Set $\X=\{X\in \cm(R) \mid \Tor_{>0}^R(X,S)=0\}$.
Then $\X$ is a thick subcategory of $\cm(R)$ containing $R,M$.
Indeed, for $X\in\cm(R)$, $\Tor_{\gg 0}^R(X,S)=0$ implies $\Tor_{>0}^R(X,S)=0$; see \cite[Theorem 9.3.6]{Avr}.
As $R$ is complete, it follows from \cite[Corollary 2.8]{S14} that $M^*$ belongs to $\X$.
Hence $\Tor_{>0}^R(M\oplus M^*,S)=0$.
It is observed from this and the $2$-periodicity of $M$ that tensoring a complete resolution of $M$ with $S$ gives a complete resolution of the $S$-module $M\otimes_R S$.
Thus $M\otimes_R S \in \Gproj S=\proj S$ as $S$ is G-regular.
The same argument as in the latter part of the proof of (1) shows that $M$ is projective over $R$.
This contradicts the equality $\cx_R M=1$.
\end{proof}

\begin{rem}
If $x$ is an $R$-regular element and $R/xR$ is G-regular, then $R$ is also G-regular \cite[Proposition 4.1]{T08}.
For a flat local homomorphism $R\to S$ of commutative noetherian local rings, if $S$ is G-regular, then $R$ is G-regular as well \cite[Proposition 4.2]{T08}.
Theorem \ref{13}(1) is a common generalization of these two facts.
\end{rem}

%%%%%%%%%%%%%%%%%%%%%%%%%%%%%%%%%%%%%%%%%%%%%
\section{Homological finiteness of thick closures}

In this section, we investigate the relationship between the strong G-regularity of a ring and the homological finiteness of certain thick subcategories.
Here, a homologically finite subcategory means one that is either contravariantly finite or covariantly finite.
These notions were first introduced by Auslander and Smal{\o} \cite{AS} in their study of which subcategories admit almost split sequences. They have since become fundamental concepts in the representation theory of algebras.
In particular, the homological finiteness of specific subcategories often characterizes important properties of rings.
For example, let $R$ be a complete commutative noetherian local ring. Auslander--Buchweitz theory \cite{AB} shows that if $R$ is Gorenstein or G-regular, then the category of finitely generated Gorenstein projective $R$-modules is contravariantly finite in $\mod R$.
Conversely, Christensen, Piepmeyer, Striuli, and Takahashi \cite{CPST08} proved that this condition characterizes the property of being either Gorenstein or G-regular.
In this section, we characterize the condition that $R$ is strongly G-regular or Gorenstein in terms of the homological finiteness of certain thick subcategories of $\mod R$.

We begin by recalling the definition of a homologically finite subcategory of an additive category.

\begin{dfn}
Let $\C$ be an additive category and $\X$ a subcategory of $\C$.
A morphism $f:X\to C$ in $\C$ is called a \textit{right $\X$-approximation} of $C$ if $X\in\X$ and for every morphism $f':X'\to C$ with $X'\in\X$, there exists a morphism $g:X'\to X$ such that $f'=fg$.
The subcategory $\X$ is said to be \textit{contravariantly finite} in $\C$ if every object of $\C$ admits a right $\X$-approximation.
Dually, the notions of \textit{left $\X$-approximations} and \textit{covariantly finite} subcategories are defined.
A subcategory is called \textit{functorially finite} if it is both contravariantly finite and covariantly finite.
%For details, we refer the reader to \cite{AS}.
\end{dfn}

%When $R$ is a Cohen--Macaulay ring, we denote by $\cm R$ the subcategory of $\mod R$ consisting of maximal Cohen--Macaulay $R$-modules. In other words, we put $\cm R=\C(R)$ when $R$ is Cohen--Macaulay.

%Let $\X$ be a subcategory of an abelian category $\mathcal{A}$. We denote by $\X^{\perp}$ the subcategory of $\mathcal{A}$ consisting of all objects $M$ satisfying $\Ext_{\mathcal{A}}^{>0}(X,M)=0$ for every $X\in\X$. Similarly, we denote by ${}^{\perp}\X$ the subcategory of $\mathcal{A}$ consisting of all objects $M$ satisfying $\Ext_{\mathcal{A}}^{>0}(M,X)=0$ for every $X\in\X$.
%With this notation, it is immediate that $\X\subset {}^{\perp}(\X^{\perp})\cap({}^{\perp}\X)^{\perp}$.
An additive category $\C$ is functorially finite in itself.
Below we present some examples of homologically finite subcategories in categories of finitely generated modules over noetherian rings.

\begin{rem}\label{contperp}
\begin{enumerate}[\rm(1)]
   \item Let $\Lambda$ be a two-sided noetherian ring.
   Then $\proj\Lambda=\add\Lambda$ is a functorially finite subcategory of $\mod\Lambda$.
   Moreover, assume that $\Lambda$ is a noetherian algebra, and let $\C$ be an additive subcategory of $\mod\Lambda$.
   If a subcategory $\X$ of $\C$ has an additive generator, that is, $\X=\add T$ for some $T\in\X$, then $\X$ is functorially finite in $\C$; see \cite[Section 4]{AS}.
   \item Let $R$ be a Cohen--Macaulay local ring with a canonical module. Then the category $\cm(R)$ of maximal Cohen--Macaulay $R$-modules is a contravariantly finite subcategory of $\mod R$. This is a consequence of Auslander--Buchweitz approximation theory \cite{AB}.
\end{enumerate}
\end{rem}

We next recall the relationship between the strong G-regularity of an artin algebra and the homological finiteness of certain subcategories.
Let $\Lambda$ be an artin algebra, not necessarily commutative. 
It was shown by Beligiannis \cite[Corollary 4.11]{B98} that $\Lambda$ is strongly G-regular if and only if it is G-regular and virtually Gorenstein.
The notion of virtually Gorenstein algebras was introduced by Beligiannis and Reiten \cite{BR}.
It is defined by an Ext-orthogonality condition between the classes of Gorenstein projective modules and Gorenstein injective modules, and provides a common generalization of Gorenstein algebras and algebras of finite representation type.
Moreover, Beligiannis and Krause \cite{BK} showed that the virtually Gorenstein property can be characterized in terms of finitely generated modules.
In fact, it is shown in \cite[Theorem 1]{BK} that $\Lambda$ is virtually Gorenstein if and only if $\thick_{\mod\Lambda}(\proj\Lambda\cup\inj\Lambda)$ is contravariantly finite in $\mod\Lambda$, if and only if it is covariantly finite in $\mod\Lambda$.
Here, $\inj\Lambda$ denotes the category of finitely generated injective $\Lambda$-modules.
Combining these results, we obtain the following theorem.

\begin{thm}[Beligiannis--Krause, Beligiannis]\label{BKB}
Let $\Lambda$ be an artin algebra. Then the following are equivalent.
\begin{enumerate}[\rm(1)]
\item $\Lambda$ is strongly G-regular.
\item $\Lambda$ is G-regular, and $\thick_{\mod\Lambda}(\proj\Lambda\cup\inj\Lambda)$ is contravariantly finite in $\mod\Lambda$.
\item $\Lambda$ is G-regular, and $\thick_{\mod\Lambda}(\proj\Lambda\cup\inj\Lambda)$ is covariantly finite in $\mod\Lambda$.
\end{enumerate}
\end{thm}

Quite recently, Dey, Liu, and Lu \cite{DLL26} proved that, for a noetherian algebra $R$ over a commutative Cohen--Macaulay ring and an artin algebra $A$, the weak Gorensteinness of $R$ and the virtual Gorensteinness of $A$, respectively, can be characterized in terms of the thick closure of the ring and its canonical dual.

Our aim in this section is to establish a higher-dimensional analogue of Theorem \ref{BKB} in the commutative setting.
To this end, we first introduce the notion of orthogonal subcategories.
Let $\X$ be a subcategory of an abelian category $\mathcal{A}$. We denote by $\X^{\perp}$ the subcategory of $\mathcal{A}$ consisting of all objects $M$ satisfying $\Ext_{\mathcal{A}}^{>0}(X,M)=0$ for every $X\in\X$. Similarly, we denote by ${}^{\perp}\X$ the subcategory of $\mathcal{A}$ consisting of all objects $M$ satisfying $\Ext_{\mathcal{A}}^{>0}(M,X)=0$ for every $X\in\X$.
With this notation, it is immediate that $\X\subseteq {}^{\perp}(\X^{\perp})\cap({}^{\perp}\X)^{\perp}$.

Under this notation, the following result will be useful. In particular, the latter assertion is Wakamatsu's lemma \cite{AR91, Wak} for resolving subcategories.

\begin{lem}\label{wak}
Let $\X$ be a subcategory of $\mod R$.
Then the following hold.
\begin{enumerate}[\rm(1)]
\item  If there exists an exact sequence $0\to Y\to X\xrightarrow{f}M\to 0$ in $\mod R$ such that $Y\in\X^\perp$ and $X\in\X$, then $f$ is a right $\X$-approximation of $M$.
\item  Suppose that $R$ is henselian and local. Let $\X$ be a resolving subcategory of $\mod R$. If $\X$ is contravariantly finite in $\mod R$, then, for any $M\in\mod R$, there exists an exact sequence $0\to Y\to X\to M\to 0$ in $\mod R$ such that $Y\in\X^{\perp}$ and $X\in\X$.
\end{enumerate}
\end{lem}

\begin{proof}
We prove assertion (1). For any $X'\in\X$, the sequence $\Hom_R(X',X)\xrightarrow{\Hom(X',f)}\Hom_R(X',M)\to\Ext_R^1(X',Y)=0$ is exact, which shows that $f$ is a right $\X$-approximation of $M$.
Assertion (2) is nothing but \cite[Lemma 3.8]{T11}.
We note that the assumption that $R$ is henselian and local is imposed to ensure that $\mod R$ is a \textit{Krull--Schmidt category}; see \cite[Section 1]{LW}.
\end{proof}

The following theorem provides a higher-dimensional analogue of Theorem \ref{BKB} for commutative rings.
Furthermore, in our setting, the assumption that the ring is G-regular is not required.
The essential part of the proof of the following theorem is based on \cite[Theorem 1.4]{T11}.

\begin{thm}\label{contfin}
Let $R$ be a henselian Cohen--Macaulay local ring with a canonical module $\omega$. Assume that $R$ is not Gorenstein. Then the following are equivalent.
\begin{enumerate}[\rm(1)]
   \item $R$ is strongly G-regular.
   \item $\thick_{\cm(R)}\{R,\omega\}$ is contravariantly finite in $\cm(R)$.
   \item $\thick_{\cm(R)}\{R,\omega\}$ is covariantly finite in $\cm(R)$.
\end{enumerate}
\end{thm}

\begin{proof}
Note that $R$ is strongly G-regular if and only if $\thick_{\Db(R)}\{R,\omega\}=\Db(R)$ by the equivalence (1)$\Leftrightarrow$(3) in Theorem \ref{str G-req quasi-dom}; see also Remark \ref{12}.
Moreover, by restricting the correspondence in Proposition \ref{2} and noting that $\c(R)=\cm(R)$ in our setting, we see that the latter condition is equivalent to $\thick_{\cm(R)}\{R,\omega\}=\cm(R)$.
Therefore, the implications (1)$\Rightarrow$(2) and (1)$\Rightarrow$(3) hold.
We prove the remaining implications.
Let $\X=\thick_{\cm(R)}\{R,\omega\}$ and set $(-)^\dagger=\Hom_R(-,\omega)$.
%Then, by the duality $(-)^\dagger:\cm(R)\xrightarrow{\sim}\cm(R)$, we obtain $\X^\dagger=\X$, and we will freely use this fact throughout the proof.
Note that we have duality $(-)^\dagger:\cm(R)\xrightarrow{\sim}\cm(R)$ and induced equality $\X^\dagger=\X$.
We can easily show that the equivalence of (2) and (3) follows from these.
Indeed, for a morphism $f$ in $\cm(R)$, it is a right $\X$-approximation of $M\in\cm(R)$ if and only if its canonical dual $f^\dagger$ is a left $\X$-approximation of $M^\dagger$.
This follows from the canonical duality, and hence (2) and (3) are equivalent.
%Indeed, assume that $\X$ is contravariantly finite in $\cm(R)$ and let $M\in\cm(R)$. Then, since $M^\dagger$ is also in $\cm(R)$, it has a right $\X$-approximation $f:X\to M^\dagger$.
%Then, for any $X'\in\X$, the homomorphism $\Hom_R({X'}^\dagger,f)$ is surjective, and the diagram
%$$
%\xymatrix{
%\Hom_R(X^\dagger,X') \ar[rrr]^-{\Hom_R(f^\dagger,X')} \ar[d]_{\cong}^{(-)^\dagger}&&&\Hom_R(M,X') \ar[d]^{\cong}_{(-)^\dagger}\\
%\Hom_R({X'}^\dagger,X) \ar[rrr]_-{\Hom_R({X'}^\dagger,f)}&&&\Hom_R({X'}^\dagger,M^\dagger).
%}
%$$
%commutes. The homomorphism $\Hom_R(f^\dagger,X')$ is surjective, and hence $f^\dagger:M\to X^\dagger$ is a left $\X$-approximation of $M$.
%This proves (2)$\Rightarrow$(3). The converse follows by a dual argument. 
Therefore, it remains to prove (2)$\Rightarrow$(1).
Fix $M\in\cm(R)$, and consider the exact sequence $0\to M^\dagger\to \omega^{\oplus n}\to(\Omega M)^\dagger\to 0$
in $\cm(R)$.
Now, since $\X$ is contravariantly finite in $\cm(R)$ by assumption and $\cm(R)$ is contravariantly finite in $\mod R$ by Remark \ref{contperp}(2), it follows that $\X$ is contravariantly finite in $\mod R$. Moreover, since $\cm(R)$ is a resolving subcategory of $\mod R$, it is easy to see that $\X$ is also a resolving subcategory of $\mod R$.
Therefore, by Lemma \ref{wak}(2), there exists an exact sequence $0\to Y\to X\to(\Omega M)^\dagger\to 0$ in $\mod R$ such that $Y\in\X^\perp$ and $X\in\X$.
Note that, in general, $\X^\perp$ is not contained in $\cm(R)$. However, since both $X$ and $(\Omega M)^\dagger$ belong to $\cm(R)$, so does $Y$.
We consider the following pullback diagram
\enlargethispage*{3\baselineskip}
$$
\xymatrix@C=2em@R=1.5em{
%\xymatrix@R-1.2pc@C+1pc{
&&0\ar[d]&0\ar[d]&\\
&&M^\dagger\ar[d]\ar@{=}[r]&M^\dagger\ar[d]&\\
0\ar[r]&Y\ar@{=}[d]\ar[r]&W\ar[d]\ar[r]&\omega^{\oplus n}\ar[d]\ar[r]&0\\
0\ar[r]&Y\ar[r]&X\ar[d]\ar[r]&(\Omega M)^\dagger\ar[d]\ar[r]&0\\
&&0&0.&
}
$$
Then the second row splits since $Y\in\X^\perp$ and $\omega^{\oplus n}\in\X$.
Hence the second column yields an exact sequence $0\to M^\dagger\to Y\oplus\omega^{\oplus n}\to X\to 0$ in $\cm(R)$.
Taking $(-)^\dagger$, we obtain an exact sequence $0\to X^\dagger\to Y^\dagger\oplus R^{\oplus n}\xrightarrow{f}M\to0$.
Moreover, $Y^\dagger\oplus R^{\oplus n}\cong (Y\oplus\omega^{\oplus n})^\dagger\in(\X^\perp\cap\cm(R))^\dagger$.
For the rest of the proof, we consider the subcategory $\C:={}^{\perp}\X\cap\cm(R)$.
Then we see that $(\X^\perp\cap\cm(R))^\dagger\subseteq\C$.
Indeed, let $T\in\X^\perp\cap\cm(R)$.
For any $X\in\X$ and any $i>0$, we have $\Ext_R^i(T^\dagger,X)\cong\Ext_R^i(X^\dagger,T)=0$, where the last equality follows from the fact that $X^\dagger\in\X^\dagger=\X$.
Thus $T^\dagger\in\C$, and the desired inclusion follows.
Also, it immediately follows that $\X\subseteq({}^{\perp}\X\cap\cm(R))^\perp=\C^\perp$.
Thus, $Y^\dagger\oplus R^{\oplus n}\in\C$ and $X^\dagger\in\X\subseteq\C^\perp$.
Hence, by Lemma \ref{wak}(1), the homomorphism $f:Y^\dagger\oplus R^{\oplus n}\to M$ is a right $\C$-approximation of $M$.
Therefore, $\C$ is contravariantly finite in $\cm(R)$, and hence also in $\mod R$.
It is easy to see that ${}^{\perp}\X$ is a resolving subcategory of $\mod R$.
Moreover, since $\cm(R)$ is also a resolving subcategory of $\mod R$, so is $\C={}^{\perp}\X\cap\cm(R)$.
Since $\C$ is closed under direct summands, it follows that $Y^\dagger\in\C$.
In particular, by the definition of $\C$, one has $\Ext_R^{>0}(Y^\dagger,R)=0$.
If $Y^\dagger$ has infinite projective dimension, then $\C=\cm(R)$ or $\mod R$ by \cite[Theorem 1.4]{T11}.
In particular, $\Omega^d k\in\C\subseteq{}^{\perp}\X$, where $d=\dim R$, and hence $\Ext_R^{>d}(k,R)=0$.
This implies that $R$ is Gorenstein, contradicting our assumption.
Therefore, $Y^\dagger$ has finite projective dimension.
However, since $Y^\dagger$ is Cohen--Macaulay, it follows that $Y^\dagger$ is free.
As both $X^\dagger$ and $Y^\dagger\oplus R^{\oplus n}$ belong to $\X$, the exact sequence $0\to X^\dagger\to Y^\dagger\oplus R^{\oplus n}\to M\to0$ implies that $M\in\X$.
Thus $\X=\cm(R)$, and $R$ is strongly G-regular.
\end{proof}

The following corollary follows from the above theorem and Remark \ref{contperp}(1).
In particular, assertion (2) is the same as \cite[Theorem 1.3{\rm(1)}]{DKLO25}; see also Corollary \ref{sgrprprng}{\rm(3)}.
Hence, the same conclusion is recovered under a slightly different assumption.

\begin{cor}\label{contfincor}
Let $R$ be a henselian Cohen--Macaulay local ring with a canonical module $\omega$. Then the following hold.
\begin{enumerate}[\rm(1)]
    \item
    If $\thick_{\cm(R)}\{R,\omega\}$ is of finite representation type, then $R$ is strongly G-regular or Gorenstein.
    \item 
    If $R$ is of finite CM-representation type, then $R$ is strongly G-regular or Gorenstein.
\end{enumerate}
\end{cor}

%%%%%%%%%%%%%%%%%%%%%%%%%%%%%%%%%%%%%%%%%
\section{Examples and counterexamples for strong G-regularity}

In this section, we show examples of rings that are strongly G-regular and of rings that are not. We first present several new constructions of (locally) dominant rings that fall outside the scope of the existing criteria. The ring arising from these constructions that is not a hypersurface is strongly G-regular. We then construct commutative artinian local rings that are G-regular and satisfy $\TR$ but are not strongly G-regular. In particular, this gives a negative answer to Question \ref{main ques}(1), even in the artinian case, and, moreover, provides negative answers to the open problems \cite[Problems A, B, and C]{Che17} for artin algebras.

The following two families of dominant rings, given in Propositions \ref{power} and \ref{T23(8.8)}, are constructed from suitable Burch rings using the preservation of dominance under various operations established in \cite{T23}.

\begin{prop}\label{power}
Let $S$ be a regular ring and $f_1,\ldots,f_n$ a regular sequence on $S$. Then $R=S/(f_1,\ldots,f_n)^s$ is locally dominant for every $s>1$.
\end{prop}

\begin{proof}
First, we assume that $S$ is local and prove that $R$ is dominant.
By \cite[Corollary 5.8]{T23}, we may assume that $S$ is complete. Set $A=S\llbracket x_1,\ldots,x_n\rrbracket/(x_1,\ldots,x_n)^s$. The natural ring homomorphism $S\to A$ is flat and local. Let $\m$ be the maximal ideal of $S$ and let $k$ be the residue field of $S$. By \cite[Proposition 5.10(1)]{T23} and \cite[Proposition 6.3(1)(g)]{KT}, $A/\m A=k\llbracket x_1,\ldots,x_n\rrbracket/(x_1,\ldots,x_n)^s$ is dominant. Since $S$ is regular, $A$ is also dominant by \cite[Proposition 7.4]{T23}. 
Put $I=(x_1-f_1,\ldots,x_n-f_n)A$. We have $A/I\cong R$. Since $A$ is Cohen--Macaulay, the equality
$$\grade I=\height I=\dim A-\dim R=\dim S-(\dim S-n)=n$$ 
holds. This shows that $x_1-f_1,\ldots,x_n-f_n$ is a regular sequence on $A$. As $s>1$, the image of $x_i-f_i$ does not belong to the square of the maximal ideal of $A/(x_1-f_1,\ldots,x_{i-1}-f_{i-1})A$. Applying \cite[Theorem 5.6]{T23} repeatedly, it follows that $R$ is dominant. 
For a general regular ring $S$, localizing at any prime ideal $\q$ of $S$ containing $f_1,\ldots,f_n$, the local case shows that $R_\q$ is dominant. This means that $R$ is locally dominant.
\end{proof}

When $S$ has finite Krull dimension, $R$ admits a dualizing complex; hence, by Corollary \ref{lqdom}, the corollary below shows that the ring $R$ constructed in Proposition \ref{power} is strongly G-regular unless $n=1$.

\begin{cor}\label{5.2}
Let $S$ be a regular ring of finite Krull dimension, let $f_1,\ldots,f_n$ be a regular sequence on $S$, and let $R=S/(f_1,\ldots,f_n)^s$ where $s>1$ and $n\ge 1$.
Then $R$ is strongly G-regular if and only if $n>1$.
\end{cor}

\begin{proof}
If $n=1$, then $R=S/(f_1^s)$ is singular and Gorenstein as $s>1$, and thus it is not strongly G-regular. Suppose $n>1$. We prove $\ng R=\spec R$. To prove this, we show that $R_\p$ is not Gorenstein for every minimal prime ideal $\p$ of $R$. For this purpose, it suffices to prove that $R$ is not Gorenstein when $(S,\m,k)$ is an $n$-dimensional regular local ring and $f_1,\ldots,f_n$ is a maximal $S$-regular sequence. Put $J=(f_1,\ldots,f_n)S$. Since $f_1,\ldots,f_n$ is a regular sequence on $S$, the associated graded ring $\operatorname{gr}_J(S)$ is isomorphic to the polynomial ring $(S/J)[X_1,\ldots, X_n]$ over $S/J$. Then $J^{s-1}/J^s$ is a free $S/J$-module of rank $\binom{n+s-2}{n-1}$. Since $J^{s-1}/J^s$ is an ideal of $R$, we have 
$$
\dim_k (0:_R \m)\ge \dim_k (0:_{J^{s-1}/J^s} \m)= \binom{n+s-2}{n-1} \dim_k (0:_{S/J} \m).
$$
As $s>1$ and $n>1$, one has $\binom{n+s-2}{n-1}\ge 2$, which implies that $R$ is not Gorenstein. 

This establishes $\ng R=\spec R$, and hence, in particular, $\ng R=\sing R$. The ring $R$, being a quotient of a regular ring of finite Krull dimension, admits a dualizing complex $D$. Since $R$ is locally dominant by Proposition \ref{power}, Corollary \ref{lqdom} gives $\thick_{\Db(R)}\{R, D\}=\Db(R)$. Finally, we conclude that $R$ is strongly G-regular by the equivalence $(1)\Leftrightarrow(3)$ in Theorem \ref{str G-req quasi-dom}; see also Remark \ref{12}.%noting that hypothesis $(\#)$ is not needed for the equivalence $(1)\Leftrightarrow(3)$ in Theorem \ref{str G-req quasi-dom}
\end{proof}

This corollary establishes strong G-regularity without requiring either isolated singularities or condition $(\#)$ as a hypothesis.
Like Proposition \ref{power}, the following proposition constructs dominant rings from a single regular sequence, but by a different method.

\begin{prop}\label{T23(8.8)}
Let $S$ be a regular local ring, $x_1,\ldots,x_n$ a regular system of parameters of $S$ where $n\ge 1$, and $a_1,\ldots,a_n, b_1,\ldots,b_n$ positive integers. Then $R=S/(x_1^{a_1},\ldots,x_n^{a_n})(x_1^{b_1},\ldots,x_n^{b_n})$ is a dominant local ring.
\end{prop}

\begin{proof}
Assume $a_1\le b_1$. We put $A=S\llbracket t\rrbracket/(t,x_2^{a_2},\ldots,x_n^{a_n})(x_1^{b_1-a_1}t, x_2^{b_2},\ldots,x_n^{b_n})$. There is a short exact sequence $0\to tA\to A\to A/tA\to 0$, and we have
\begin{align*}
&A/tA\cong S/(x_2^{a_2},\ldots,x_n^{a_n})(x_2^{b_2},\ldots,x_n^{b_n}) \ \text{and} \\
& tA\cong A/(0:_A t)\cong S\llbracket t\rrbracket/(x_1^{b_1-a_1}t, x_2^{b_2},\ldots,x_n^{b_n}, x_1^{b_1-a_1}x_2^{a_2},\ldots, x_1^{b_1-a_1}x_n^{a_n}).
\end{align*}
For any $c>0$, the image of $x_1^c-t$ does not belong to the square of the maximal ideal of $A$.
Then $x_1^c-t$ is regular on both $A/tA$ and $tA$. Indeed, for ideals $C=(x_2^{a_2},\ldots,x_n^{a_n})$ and $D=(x_2^{b_2},\ldots,x_n^{b_n})$ of $S\llbracket t\rrbracket$, we have 
$$(x_1^{b_1-a_1}t, D, x_1^{b_1-a_1}C)=(x_1^{b_1-a_1}, D) \cap (t, C, D)$$
in $S\llbracket t\rrbracket$. 
A primary decomposition of the right-hand side shows that $x_1^c-t$ is regular on $tA$.
Applying \cite[Theorem 5.6]{T23} twice, with $c=a_1$ and $c=1$, shows that the following are equivalent: $R\cong A/(x_1^{a_1}-t)A$ is dominant, $A$ is dominant, and $A/(x_1-t)A\cong S/(x_1,x_2^{a_2},\ldots,x_n^{a_n})(x_1^{b_1-a_1+1},x_2^{b_2},\ldots,x_n^{b_n})$ is dominant. Thus we may assume that $a_1=1$.

Repeating the above operation, we may assume that, for every $i=1,\ldots,n$, either $a_i=1$ or $b_i=1$. In particular, by symmetry, we may assume $a_1=1$. Then $x_1^{b_1}$ belongs to $(I:_S \m)$, where $\m=(x_1,\ldots,x_n)$ and $I=(x_1^{a_1},\ldots,x_n^{a_n})(x_1^{b_1},\ldots,x_n^{b_n})S$. In fact, fix $i=1,\ldots,n$. If $a_i=1$, then $x_ix_1^{b_1} \in I$ is obvious; otherwise, $x_ix_1^{b_1}=(x_i^{b_i}x_1^{a_1})x_1^{b_1-1}\in I$ since $a_1=1=b_i$. Hence $x_1^{b_1+1}$ belongs to $\m(I:_S \m)$. On the other hand, $x_1^{b_1+1}$ is not in $\m I$ since it is one of the minimal generators of $I$. It follows that $R=S/I$ is Burch, and hence it is dominant by \cite[Proposition 5.10]{T23}.
\end{proof}

As in Corollary \ref{5.2}, the rings constructed in this proposition are strongly G-regular unless they are hypersurfaces.

\begin{cor}
In the notation of Proposition \ref{T23(8.8)}, the ring $R$ is strongly G-regular if and only if $n>1$.
\end{cor}

\begin{proof}
If $n=1$, then $R=S/(x_1^{a_1+b_1})$ is singular and Gorenstein, and thus it is not strongly G-regular. Suppose $n>1$. We set $I=(x_1^{a_1},\ldots,x_n^{a_n})(x_1^{b_1},\ldots,x_n^{b_n})S$. 
Let $\m$ be the maximal ideal of $S$ and let $k$ be the residue field of $S$. For each $i=1,\ldots,n$, set
$$
u_i=x_i^{a_i+b_i-1}\prod_{j\ne i}x_j^{\min\{a_j,b_j\}-1}.
$$
We first observe that $u_i\notin I$. Indeed, if $p\ne i$, then the exponent of $x_p$ in $u_i$ is less than $a_p$; if $q\ne i$, then the exponent of $x_q$ in $u_i$ is less than $b_q$; and the exponent of $x_i$ in $u_i$ is less than $a_i+b_i$. Hence no monomial generator $x_p^{a_p}x_q^{b_q}$ of $I$ divides $u_i$.
We show that $\m u_i\subseteq I$. Clearly, $x_i u_i\in I$. Let $j\ne i$. If $a_j\le b_j$, then $x_j u_i$ is divisible by $x_j^{a_j}x_i^{b_i}$, while if $b_j<a_j$, then $x_j u_i$ is divisible by $x_i^{a_i}x_j^{b_j}$. Hence $x_j u_i\in I$ for every $j=1,\ldots,n$. 
It follows that $u_1,\ldots,u_n$ belong to $(I:_S\m)$, and their images in $(0:_R\m)$ are linearly independent over $k$. Indeed, for all $i=1,\ldots,n$, since $u_i$ is divisible neither by any $u_j$ with $j\ne i$ nor by any monomial in $x_1,\ldots,x_n$ that belongs to $I$, \cite[Corollary 3]{KiSt} shows that $u_i\notin I+ \sum_{j\ne i} S u_j$. Thus
$$
\dim_k \operatorname{Soc} R=\dim_k (0:_R\m)\ge n>1,
$$
which shows that $R$ is not Gorenstein.
As $R$ is an artinian dominant local ring with $\ng R=\sing R=\spec R$, the implication (7)$\Rightarrow$(1) of Theorem \ref{str G-req quasi-dom} yields that $R$ is strongly G-regular.
\end{proof}

\begin{rem}
\begin{enumerate}[(1)]
\item Proposition \ref{power} provides a related positive result that complements \cite[Proposition 7.7]{T23}. Proposition \ref{T23(8.8)} concerns a special type of product of two ideals generated by regular sequences, as in \cite[Proposition 8.8]{T23}, but does not require the concatenation of the two sequences to be regular. In particular, these propositions give families of dominant rings that are not fully covered by previous results, such as those in \cite{T23,KT}.

\item Whether every Golod local ring is dominant is an interesting open problem; see \cite[Question 9.4]{T23}. Hence it is also natural to ask whether every Golod local ring admitting a dualizing complex is quasi-dominant. Since every Cohen--Macaulay local ring of minimal multiplicity is Golod by \cite[Example 5.2.8]{Avr}, Example \ref{not deform sgr} provides an example of a Golod ring that is not strongly G-regular. On the other hand, it is known that every Golod local ring that is not a hypersurface is G-regular; see \cite[Examples 3.5]{AM}. 

The G-regularity of the rings appearing in Corollary \ref{5.2} can also be deduced from Golod theory. Indeed, the quotient of a regular local ring by the $s$-th power of a complete intersection ideal is Golod for every $s\ge2$; see \cite[Corollary 3.3]{GS} for instance. The conclusion obtained in this paper is stronger: Proposition \ref{power} shows that these rings are locally dominant, and Corollary \ref{5.2} shows that they are strongly G-regular when $n>1$. It follows from \cite[Corollary 1.2]{DDS22} that the quotient of a polynomial ring in three variables by the product of two proper monomial ideals is Golod. Proposition \ref{T23(8.8)} shows that, when the monomial ideals are parameter ideals, its completion is dominant.
\end{enumerate}
\end{rem}

The examples of commutative rings that are G-regular but not strongly G-regular given in Corollary \ref{11} arise from the failure of G-regularity to be preserved under localization. However, any counterexample to Question \ref{main ques}(1) obtained in this way necessarily has positive Krull dimension. Question \ref{main ques}(1) has its origin in \cite[Problem C]{Che17}, which was posed for artin algebras and had remained open in that setting. Thus, a counterexample relying on loci or localization would not resolve the original problem in its intended setting. We therefore study strong G-regularity for artinian rings by developing a method for constructing nontrivial totally acyclic complexes. The following proposition gives a sufficient condition for a ring not to be strongly G-regular.

\begin{prop}\label{twosqzeromaps}
Let $G$ be a projective $R$-module, and let $a,b:G\to G$ be $R$-endomorphisms satisfying $a^2=0=b^2$ and $ab=ba$. Assume that the following two conditions hold.
\begin{enumerate}[{\rm (1)}]
\item There is an endomorphism $g$ of the additive group $G$ such that $ag+ga=\id_G$ and $bg=gb$.
\item For every free $R$-module $H$, there exists an endomorphism $h$ of the additive group $\Hom_R(G,H)$ such that $b^*h+hb^*=\id_{\Hom_R(G,H)}$ and $a^*h=ha^*$, where $a^*=\Hom_R(a,H)$ and $b^*=\Hom_R(b,H)$.
\end{enumerate}
Set $F=G\oplus G$ and define an $R$-endomorphism $c$ of $F$ by
$$
c=
\begin{pmatrix}
a&b\\
b&-a
\end{pmatrix}.
$$
Then the one-periodic complex $X:=(\cdots\xrightarrow{c}F\xrightarrow{c}F\xrightarrow{c}F\xrightarrow{c}\cdots)$ is totally acyclic. In particular, if $\Im c$ is not projective, then $R$ is not strongly G-regular.
\end{prop}

\begin{proof}
The assumptions $a^2=0=b^2$ and $ab=ba$ yield
$$
c^2=
\begin{pmatrix}
a^2+b^2&ab-ba\\
ba-ab&a^2+b^2
\end{pmatrix}
=0.
$$
Using $ag+ga=\id_G$ and $bg=gb$, we obtain
$$
c
\begin{pmatrix}
g&0\\
0&-g
\end{pmatrix}
+
\begin{pmatrix}
g&0\\
0&-g
\end{pmatrix}c
=
\begin{pmatrix}
ag+ga&-bg+gb\\
bg-gb&ag+ga
\end{pmatrix}
=\id_F.
$$
This equality shows that $X$ is acyclic.

Let $H$ be a free $R$-module. Under the natural isomorphism $\Hom_R(F,H)\cong\Hom_R(G,H)\oplus\Hom_R(G,H)$, the map $c^*=\Hom_R(c,H)$ is represented by
$$
c^*=
\begin{pmatrix}
a^*&b^*\\
b^*&-a^*
\end{pmatrix}.
$$
It follows from $b^*h+hb^*=\id_{\Hom_R(G,H)}$ and $a^*h=ha^*$ that
$$
c^*
\begin{pmatrix}
0&h\\
h&0
\end{pmatrix}
+
\begin{pmatrix}
0&h\\
h&0
\end{pmatrix}
c^*
=
\begin{pmatrix}
b^*h+hb^*&a^*h-ha^*\\
-a^*h+ha^*&b^*h+hb^*
\end{pmatrix}
=\id_{\Hom_R(F,H)}.
$$
Hence $\Hom_R(X,H)$ is acyclic. For every projective $R$-module $L$, the complex $\Hom_R(X,L)$ is a direct summand of $\Hom_R(X,H')$ for some free $R$-module $H'$, and is therefore acyclic. Thus $X$ is totally acyclic. If $\Im c$ is not projective, then $\Im c\in\GProj R\setminus\Proj R$, and hence $R$ is not strongly G-regular.
\end{proof}

The idea behind Proposition \ref{twosqzeromaps} is to realize acyclicity and dual acyclicity separately and then combine them. Since $g$ is only an endomorphism of the underlying additive group of $G$, the equality $ag+ga=\id_G$ does not necessarily make the one-periodic complex determined by $a$ split exact as a complex of $R$-modules. It does, however, imply that this complex is acyclic. Likewise, for every free $R$-module $H$, the equality $b^*h+hb^*=\id_{\Hom_R(G,H)}$ implies that the one-periodic complex determined by $b^*$ is acyclic, although it need not be split exact as a complex of $R$-modules.
Proposition \ref{twosqzeromaps} shows that, when these two pieces of acyclicity data satisfy the appropriate compatibility conditions, they can be combined to construct a totally acyclic complex. The key idea of the following theorem is that, when the pairs $(a,g)$ and $(b,h)$ are constructed using the two different tensor factors, and the tensor product structure naturally ensures the required commutativity.

\begin{thm}\label{hanrei}
Let $k$ be a field, and let $(A,\m,k)$ and $(B,\n,k)$ be artinian local $k$-algebras that are not fields. Assume that $\m^2=0=\n^2$, and put $R=A\otimes_kB$. Then $R$ is not strongly G-regular. Moreover, if $\dim_k\m\ge2$ and $\dim_k\n\ge2$, then $R$ is G-regular and satisfies $\TR$ but is not strongly G-regular.
\end{thm}

\begin{proof}
Put $p=\dim_k\m$ and $q=\dim_k\n$, and fix $k$-bases $\{x_1,\ldots,x_p\}$ of $\m$ and $\{y_1,\ldots,y_q\}$ of $\n$. We also fix bijections
$$
\lambda:\mathbb{N}\times\{1,\ldots,p\}\xrightarrow{\cong}\mathbb{N}
\quad\text{and}\quad
\mu:\mathbb{N}\times\{1,\ldots,q\}\xrightarrow{\cong}\mathbb{N}.
$$
Let $G$ be the free $R$-module with basis $\{e_{r,s}\}_{(r,s)\in\mathbb{N}^2}$.
Define $R$-endomorphisms $a,b$ of $G$ by
$$
a(e_{\lambda(r,i),s})=x_ie_{r,s}
\quad\text{and}\quad
b(e_{r,s})=\sum_{j=1}^q y_je_{r,\mu(s,j)}.
$$
Since $\lambda$ is bijective, the first formula defines $a$ on every basis element of $G$. The equalities $\m^2=0=\n^2$ give $a^2=0=b^2$, and $ab=ba$ because $a$ and $b$ act on different indices; see
$$
ba(e_{\lambda(r,i),s})=\sum_{j=1}^q x_iy_je_{r,\mu(s,j)}=ab(e_{\lambda(r,i),s}).
$$

We first verify the condition (1) in Proposition \ref{twosqzeromaps}. Since $R=B\oplus x_1B\oplus\cdots\oplus x_pB$ as a free $B$-module, 
$\{e_{r,s}, x_1e_{r,s}, \ldots, x_pe_{r,s} \mid (r,s)\in\mathbb{N}^2\}$ is a $B$-basis of $G$. So we can define a $B$-linear map $g:G\to G$ by 
$$
g(e_{r,s})=0
\quad\text{and}\quad
g(x_ie_{r,s})=e_{\lambda(r,i),s}.
$$
A direct check on the corresponding $B$-basis of $G$ gives $ag+ga=\id_G$.
Moreover, as $g$ changes only the first index, whereas $b$ changes only the second index, we obtain $gb=bg$ as $B$-linear maps; see
\begin{align*}
&gb(e_{r,s})=\sum_{j=1}^q y_j g(e_{r,\mu(s,j)})=0=bg(e_{r,s}), \\
&gb(x_ie_{r,s})=\sum_{j=1}^q y_j g(x_ie_{r,\mu(s,j)})=\sum_{j=1}^q y_je_{\lambda(r,i),\mu(s,j)}=b(e_{\lambda(r,i),s})=bg(x_ie_{r,s}).
\end{align*}

We next verify (2). Let $H=R^{\oplus\Lambda}$ be a free $R$-module and put $K=A^{\oplus\Lambda}$. As a free $A$-module, we have $H=K\otimes_k B=(K\otimes_k k)\oplus (K\otimes_k ky_1)\oplus\cdots\oplus (K\otimes_k ky_q)$. 
Thus, for every $\phi\in\Hom_R(G,H)$, there are unique elements $u_{r,s},v_{r,s,1},\ldots,v_{r,s,q}\in K$ such that
$$
\phi(e_{r,s})=(u_{r,s}\otimes 1)+\sum_{j=1}^q (v_{r,s,j} \otimes y_j)
$$
for all $(r,s)\in\mathbb{N}^2$. Since $\mu$ is bijective and $\{e_{r,s}\}_{(r,s)\in\mathbb{N}^2}$ is an $R$-basis of $G$, the formulas
$$
h(\phi)(e_{r,\mu(s,j)})=v_{r,s,j}\otimes 1
$$
define an $R$-homomorphism $h(\phi): G\to H$. This correspondence $\phi\mapsto h(\phi)$ yields an additive endomorphism $h$ of $\Hom_R(G,H)$.
Set $a^*=\Hom_R(a,H)$ and $b^*=\Hom_R(b,H)$. Since $\n^2=0$, we have
$$
(b^*\phi)(e_{r,s})=\phi(b(e_{r,s}))=\sum_{j=1}^q y_j\phi(e_{r,\mu(s,j)})=\sum_{j=1}^q u_{r,\mu(s,j)}\otimes y_j.
$$
It follows from the definition of $h$ that
$$
(b^*h(\phi))(e_{r,s})=h(\phi)(b(e_{r,s}))=\sum_{j=1}^q y_j h(\phi)(e_{r,\mu(s,j)})=
\sum_{j=1}^q v_{r,s,j}\otimes y_j \quad\text{and}\quad h(b^*(\phi))(e_{r,\mu(s,j)})=u_{r,\mu(s,j)} \otimes 1.
$$
Therefore, for every $R$-homomorphism $\phi$, the maps $\phi$ and $(b^*h+hb^*)(\phi)$ agree on an $R$-basis $\{e_{r,s}\}_{(r,s)\in\mathbb{N}^2}$ of $G$, and hence they agree as $R$-homomorphisms. This means $b^*h+hb^*=\id_{\Hom_R(G,H)}$.

For all $r,s\in\mathbb{N}$, $1\le i\le p$, and $1\le j\le q$, the $K$-coefficient of $y_j$ in
$$
(a^*\phi)(e_{\lambda(r,i),s})=\phi(a(e_{\lambda(r,i),s}))=\phi(x_ie_{r,s})=x_i\phi(e_{r,s})=(x_iu_{r,s}\otimes 1)+\sum_{k=1}^q (x_iv_{r,s,k} \otimes y_k)
$$
is $x_iv_{r,s,j}$. By the definition of $h$, we have $(ha^*(\phi))(e_{\lambda(r,i),\mu(s,j)})=x_iv_{r,s,j} \otimes 1$. On the other hand, 
$$
(a^*h(\phi))(e_{\lambda(r,i),\mu(s,j)})=h(\phi)(a(e_{\lambda(r,i),\mu(s,j)}))=h(\phi)(x_ie_{r,\mu(s,j)})=x_ih(\phi)(e_{r,\mu(s,j)})=x_iv_{r,s,j}\otimes 1.
$$
The bijectivity of $\lambda$ and $\mu$ now yields $a^*h(\phi)=ha^*(\phi)$, and thus $a^*h=ha^*$.

Both conditions in Proposition \ref{twosqzeromaps} are satisfied.
Set $F=G\oplus G$, let
$$
c=
\begin{pmatrix}
a&b\\
b&-a
\end{pmatrix}:F\to F,
$$
and put $M=\Im c$. Proposition \ref{twosqzeromaps} shows that $M$ is Gorenstein projective.
The ring $R$ is local with maximal ideal $I=\m R+\n R$, and one has $I^2\ne 0$ and $I^3=0$. 
By the definition of $a$ and $b$, we have $0\ne M\subseteq IF$, which implies that $I^2M\subseteq I^3F=0$. Therefore, $M$ is not projective, and $R$ is not strongly G-regular.

Finally, $\dim_k (I/I^2)=p+q$ and $\dim_k \operatorname{Soc} R=\dim_k I^2=pq$. 
If $p,q\ge2$, then $pq+1\ne p+q$ as
$$
(pq+1)-(p+q)=(p-1)(q-1)>0.
$$
By \cite[Example 2.5(a)]{AV} and Remark \ref{5}(2), $R$ is G-regular and satisfies $\TR$.
\end{proof}

\begin{rem}
\begin{enumerate}[(1)]
\item When $\dim_k\m=2=\dim_k\n$, it was already shown in \cite[Remark 9.8]{T23} that the ring in Theorem \ref{hanrei} is G-regular but not dominant. The above theorem is the first to show that this ring is not strongly G-regular. This shows that, for artinian local rings, strong G-regularity in the sense of this paper is strictly stronger than strong G-regularity in the sense of \cite{AV}; see Remark \ref{5}.

\item Moreover, over a commutative G-regular local ring, the ring itself is the only finitely generated indecomposable Gorenstein projective module up to isomorphism. Thus, the ring in Theorem \ref{hanrei} is CM-finite in the sense of \cite{Che17}, but admits a Gorenstein projective module that cannot be expressed as a direct sum of finitely generated Gorenstein projective modules. Therefore, it provides a counterexample to \cite[Problem A]{Che17}. As noted immediately before the statement of \cite[Problem B]{Che17}, it is equivalent to \cite[Problem A]{Che17}; hence, \cite[Problem B]{Che17} is also settled in the negative. The rings in Theorem \ref{hanrei}, which provide these counterexamples, are also weakly Gorenstein in the sense of Ringel and Zhang \cite{RZ}.
\end{enumerate}
\end{rem}

\begin{cor}\label{counterex}
Even for artinian local rings, Question \ref{main ques}(1) has a negative answer. In particular, all of the problems (A), (B), and (C) given in Section 1 (the same as \cite[Problems A,B,C]{Che17}) have negative answers.
\end{cor}

\begin{ac}
The authors thank Toshinori Kobayashi for giving valuable comments.
Kaito Kimura was supported by JSPS Overseas Research Fellowships.
Yuki Mifune was supported by Grant-in-Aid for JSPS Fellows 25KJ1386.
Yuya Otake was supported by Grant-in-Aid for JSPS Fellows 26KJ0067.
Ryo Takahashi was supported by JSPS Grant-in-Aid for Scientific Research 26K00599.
\end{ac}

%%%%%%%%%%%%%%%%%%%%%%%%%%%%%%%%%%%%%%%%%

\end{document}